\documentclass[leqno,11pt,a4paper]{amsart}

\usepackage{amssymb,latexsym}
\usepackage{graphicx}
\usepackage{hyperref}
\usepackage{amsmath}
\usepackage{amsfonts}
\usepackage{amssymb}
\usepackage{float}%

\makeatletter
\@namedef{subjclassname@2020}{%
  \textup{2020} Mathematics Subject Classification}
\makeatother

\theoremstyle{plain} 
\newtheorem{theorem}{\bf Theorem}[section]

\newtheorem{corollary}[theorem]{\bf Corollary}

\newtheorem{proposition}[theorem]{\bf Proposition}
\newtheorem{conjecture}[theorem]{\bf Conjecture}

\theoremstyle{definition} 
\newtheorem{definition}[theorem]{\bf Definition}
\newtheorem{remark}[theorem]{\bf Remark}

\newtheorem{example}[theorem]{\bf Example}

\newcommand{\val}{\operatorname{val}}

\newcommand{\newemptyset} {\varnothing}

\newcommand{\interm}{\operatorname{in}}

\newcommand{\q}[1]{``#1''}

\makeatletter
\DeclareRobustCommand{\Udots}{%
  \vcenter{\offinterlineskip
    \halign{%
      \hbox to .8em{##}\cr
      \hfil.\cr\noalign{\kern.2ex}
      \hfil.\hfil\cr\noalign{\kern.2ex}
      .\hfil\cr}%
  }%
}
\makeatother

\title[The characteristic 2 anisotropicity of simplicial spheres]{The 
          characteristic 2 anisotropicity of simplicial spheres} 
 
\author[Stavros~A.~Papadakis]{Stavros~Argyrios~Papadakis}
\address{Stavros~Argyrios~Papadakis\\ Department of Mathematics \\
University of Ioannina\\
Ioannina, 45110 \\
Greece}
\email{spapadak@uoi.gr}

\author[Vasiliki Petrotou]{Vasiliki Petrotou}
\address{Vasiliki Petrotou\\ Department of Mathematics \\
University of Ioannina\\
Ioannina, 45110 \\
Greece}
\email{v.petrotou@uoi.gr}

\begin{document}


\begin{abstract} Assume $D$ is a simplicial sphere, and  $k_1$ 
is a field.
We say that   $D$ is generically anisotropic over $k_1$
if,  for a certain purely transcendental  field extension $k$ of $k_1$,
a certain Artinian reduction $A$  of the Stanley-Reisner ring $k[D]$ 
 has the following  property: All nonzero homogeneous elements $u \in A$ 
of degree less or equal to $(\dim D +1)/2$ have nonzero square. 
We prove, using suitable differential operators, 
that, if the field $k_1$ has characteristic $2$, then every 
simplicial sphere $D$ is generically anisotropic over $k_1$. As an application,
we give a second proof of a recent result of Adiprasito, known as 
McMullen's g-conjecture for simplicial spheres. 
We  also prove that the simplicial spheres
of dimension $1$ are generically anisotropic over any field $k_1$.
\end{abstract}

\maketitle

\tableofcontents

\section{Introduction} 

The motivation for the present work was  McMullen's g-conjecture for simplicial spheres
\cite{BL, St3, Swa}. 
Recently, a proof of the conjecture, due to Adiprasito, has appeared \cite{A1,A2}.
 
Our approach is based on the well-known result  that to prove the g-conjecture 
for a simplicial sphere $D$ it is enough to find a  field  $k$ such that
the Stanley-Reisner ring (also known as face ring) $k[D]$ has the Weak Lefschetz
Property.   

Instead of working directly with the Weak Lefschetz Property, we introduce
in Definition~\ref{defn!genericanisotropic} the notion of
a simplicial sphere $D$ being generically anisotropic over a field $k_1$.
This means that for a certain purely transcendental  field extension $k$ of $k_1$,
a certain Artinian reduction $A$
 of the Stanley-Reisner ring $k[D]$ 
 has the following 
property: All nonzero homogeneous elements $u \in A$ of degree less or equal
to  $(\dim D +1)/2$ have nonzero square. 
We remark that $A$ is the generic Artinian reduction of $k_1[D]$ in the sense of 
Definition~\ref{dfn!genericartinianreduction}. 

Using some  ideas and results of Swartz \cite{S1}, we show in
Theorem~\ref{thm!genericanisotropicityimplieswlpnonprecise} that if the
suspension $S(D)$ of $D$ is generically anisotropic over $k_1$, 
then the Stanley-Reisner ring $k[D]$ has the Weak Lefschetz Property. 

We establish two results related to generic anisotropicity.  In 
Theorem~\ref{thm!1dimisanisotropic},  we prove that 
a simplicial sphere of dimension $1$ is generically anisotropic
over any  (finite or infinite) field  $k_1$.
In Theorem~\ref{thm!anisotropicityinchar2},
we prove that 
over any   (finite or infinite)  
field of characteristic $2$, every simplicial sphere
is  generically anisotropic.   The question of
generic  anisotropicity of simplicial spheres of dimension 
$ \geq 2$ over a field of characteristic not equal
to $2$ remains open. 

For both theorems, an important step is to  understand
(part of) the multiplicative structure of $A$
in terms of rational functions on a certain transcendence basis of 
the field extension $k_1 \subset  k$. The key
results here are Proposition~\ref{prop!psiinth1dimcase},
which works in all characteristics but only for simplicial
spheres of dimension $1$, and Theorem~\ref{thm!aboutlsquares}
which is valid in any dimension but only in 
characteristic $2$.  An interesting open question
is to establish a version of Theorem~\ref{thm!aboutlsquares}
valid in all characteristics.

Using Proposition~\ref{prop!psiinth1dimcase},
we establish generic anisotropicity
in all characteristics  for simplicial
spheres of dimension $1$ 
by an initial terms argument, see
Proposition~\ref{prop!vanishingofdi}.

We now describe the way we use Theorem~\ref{thm!aboutlsquares}
to prove generic anisotropicity in characteristic $2$ and all
dimensions. 
 We introduce, in Section~\ref{sec!usingthefdifferentialoperatorsto},
$(\dim D+1)$-th order differential operators $\partial_{\sigma}$ 
and $\partial_{p, \sigma}$,
associated to certain faces $\sigma, \sigma \cup\{p\}$ of $D$.
In Sections~\ref{sec!diffoperatornodd}, \ref{sec!diffoperatorneven} 
and~\ref{sec!diffoperatoridentities} we study the differential operators in some detail,
and prove Theorem~\ref{theorem!computingdifferoperators},
which states  identities related to the differentiation
of the product of the maximal minors of certain matrices.
The theorem is used to prove
the key  Propositions~\ref{prop!differoddnpfofmainthm} and
\ref{prop!differforevennpfofmainthm}. 
The propositions  imply  Corollaries~\ref{cor!moioimofdf} 
and~\ref{cor!pdgfsaffsa}, and the corollaries imply
Theorem~\ref{thm!anisotropicityinchar2}.
We remark that even though the differential operators
can be defined in any characteristic,  most of their
properties we need are true only in characteristic $2$.

  Finally, combining  Theorem~\ref{thm!anisotropicityinchar2} with
Theorem~\ref{thm!genericanisotropicityimplieswlpnonprecise} 
we get  a second  proof of McMullen's g-conjecture for simplicial spheres 
in  Theorem~\ref{thm!secondproofofgconjecture}.

One key question is where and how one gets the positivity  or,
maybe better, nondegeneracy needed for anisotropicity
and the Lefschetz Properties.
In the present work, we get it from the well-known nondegenerate
pairings, valid in any  Gorenstein Artinian graded algebra, 
described in Remark~\ref{rem!poincare_duality_for_gor}. 
The actual places where we use the nondegeneracy
of  the pairings to achieve anisotropicity 
are in the proofs  of Corollaries~\ref{cor!moioimofdf} and~\ref{cor!pdgfsaffsa}.

It is well-known that anisotropicity is not well-behaved under field extensions.
For example, if $m \geq 2$, then  any $m$-dimensional positive
definite  symmetric bilinear form over the real 
numbers is anisotropic, but, if we make a field extension to  the 
complex numbers, anisotropicity is lost.  
In addition, over an algebraically closed field
any symmetric bilinear form on a vector space 
of dimension $\geq 2$ is not anisotropic.
This (partially) explains the need to introduce the
field extension $k$ of $k_1$  and  the Artinian $k$-algebra $A$.  As 
Section~\ref{sec!1dimensional} suggests, this way we enter into a setting 
with an arithmetic over a  function field flavour.

Since  $A$ is a graded Artinian Gorenstein algebra
canonically associated  with the pair $(D,k_1)$,  it is, perhaps,
not surprising that it contains useful information. 
On the other hand, we found really surprising  that the peculiarities
of characteristic $2$, such as the property that
partial derivatives are linear  over  the subfield $(k_1)^2$  of $k_1$
(see Remark~\ref{rem!aboutcharacterstic2fifferentialoperators}),
play a key role in our arguments.  Maybe this  hints  to a closer
connection  between the combinatorial properties of a  simplicial complex 
$D$ and the  properties of the generic Artinian reduction of the Stanley-Reisner
ring  of $D$ over the finite field with two element $\mathbb{Z}/(2)$.

The file \ \href{http://users.uoi.gr/spapadak/anisotropicitym2codev1.txt}{http://users.uoi.gr/spapadak/anisotropicitym2codev1.txt} \
contains Macaulay2~\cite{GS} code related to the present paper.

\section{Notation} \label{sec!notations}

Assume $k$ is a field of arbitrary characteristic.
All graded $k$-algebras will be commutative, Noetherian and of the form 
$G=\oplus_{i \geq0} G_{i}$ with $G_{0}=k$ and $\dim_{k} G_{i} < \infty$ for all
$i$. The $k$-algebra
$G$ is called standard graded if it is generated, as a $k$-algebra, by $G_{1}$. 
We denote by $\; \dim G \;$ the Krull dimension of $G$.

Assume $F$ is an Artinian graded $k$-algebra. 
There exists a  largest  integer  $d$ such that the $d$-th
graded part $F_d$ is nonzero, and we call $d$ the socle degree of $F$.
An element  $\omega\in F_{1}$ is called a Weak Lefschetz element
if, for all $i \geq 0$, the multiplication by $\omega$ map $F_{i} \to F_{i+1}$ 
is of maximal rank,
which means that it is injective or surjective (or both). 
We say that  $F$ has the  Weak
Lefschetz Property  if there exists a Weak Lefschetz element
 $\omega\in F_{1}$.

Assume $F$ is an Artinian Gorenstein graded $k$-algebra
of socle degree $d$.  An element  $\omega\in F_{1}$ is called a 
Strong Lefschetz element if, for
all $i$ with $0 \leq 2i \leq d$,
the multiplication by $\omega^{d-2i}$ map $F_{i} \to F_{d-i}$ is bijective.
We say that  $F$ has the Strong
Lefschetz Property  if there exists a Strong Lefschetz element
 $\omega\in F_{1}$.

We say that a standard graded $k$-algebra $G$ with positive  Krull dimension 
has the 
Weak Lefschetz Property if it is Cohen-Macaulay, the field $k$ is infinite and 
for Zariski general  homogeneous  degree $1$  elements  $f_{1}, \dots,
f_{\dim G}$ of $G$
the Artinian $k$-algebra $G/(f_{1}, \dots, f_{\dim G})$
has the Weak Lefschetz Property.

We say that a standard graded $k$-algebra $G$ with positive  Krull dimension 
has the 
Strong Lefschetz Property if it is Gorenstein,  the field $k$ is infinite and 
for Zariski general  homogeneous  degree $1$  elements  $f_{1}, \dots,
f_{\dim G}$ of $G$
the Artinian $k$-algebra $G/(f_{1}, \dots, f_{\dim G})$
has the Strong  Lefschetz Property. 
Good general references for the Weak and Strong Lefschetz Properties 
are \cite{HMetal,MiNa}.

\begin{remark} \label{rem!poincare_duality_for_gor} 
We will use the following well-known fact, see for
example  \cite[Theorem 2.79]{HMetal}. Assume
 $F=\oplus_{i=0}^{d} F_{i}$ with $F_{d} \not = 0$ is a standard graded Gorenstein
Artinian $k$-algebra. Then $F_{d}$ is $1$-dimensional, and, for all $i$ with $0
\leq i \leq d$, the multiplication map $F_{i} \times F_{d-i} \to F_{d} \cong k$
is a perfect pairing.  As a consequence, given $i,j$ with $0 \leq i \leq j \leq d$ and 
a nonzero element $u \in F_{i}$,  there exists $w \in F_{j-i}$ such that
$uw \not= 0$.  The reason is that by the perfect pairing property 
there exists $w_1 \in F_{d-i}$ such that $uw_1 \not= 0$, and since
$F$ is standard graded, $w_1$ is a sum of products of elements of
$F_{j-i}$ with elements of $F_{d-j}$.
\end{remark}

We will use the definitions of \cite[Section~5.1]{BH} for the notions 
of simplicial complex, vertices, faces, facets and  dimension of a simplicial complex, 
and the notion of 
Stanley-Reisner ideal (also known as face ideal) and
Stanley-Reisner  ring (also known as face ring)
of a simplicial complex over the field $k$.  If $n \geq 1$ is
an integer,  a simplicial sphere
of dimension $n$ is a simplicial complex $D$ of dimension $n$ such that
there exists a geometric realization of $D$, in the sense of 
\cite[Definition~5.2.8]{BH}, which is homeomorphic to the sphere $S^{n}$.
An additional important  general reference  is \cite{St1}.

\subsection{The generic Artinian reduction of an algebra} \label{subs!genericArtinianReduction}

The following is a useful construction that will appear a number of times
in the present paper.

Assume $m \geq 1$ and  $k_1$ is a field.  We consider the polynomial 
ring  $k_1[x_1, \dots , x_m]$, where the degree of the variable $x_i$ is equal to $1$,
for all $1 \leq i \leq m$.
Assume   $I \subset k_1[x_1, \dots , x_m]$ 
is a homogeneous ideal. We denote by $d$ the  
Krull dimension of the quotient ring  $k_1[x_1, \dots , x_m]/I$.
We assume  $d \geq 1$, and denote 
by $k$ the field of fractions  of the polynomial ring
\[
             k_1 [ a_{i,j} :  1 \leq i \leq d,  \;  1 \leq j \leq m ].
\]
For $1 \leq i \leq d$, we set 
\[
     f_i = \sum_{j=1}^{m} a_{i,j} x_j.
\]

\begin{definition}  \label{dfn!genericartinianreduction}
We define the generic Artinian reduction of  $k_1[x_1, \dots , x_m]/I$ to be the
Artinian $k$-algebra
\[
     k[x_1, \dots , x_m]/ ( (I) + (f_1, \dots , f_d)),      
\]
where $(I)$ denotes the ideal of $k[x_1, \dots , x_m]$ generated by $I$.
\end{definition}

\section{Statement of the main theorem} \label{sec!statementofthemaintheorem}

Assume  $n \geq 1$ is an integer and $D$ is a simplicial sphere
of dimension $n$ with vertex set $\{1, \dots , m  \}$.
Assume $k_1$ is any field and denote by $k$ the field of fractions of the polynomial
ring
\[
             k_1 [ a_{i,j} :  1 \leq i \leq n+1,  \;  1 \leq j \leq m ].
\]
We define the polynomial ring  $R = k[x_1, \dots , x_m]$,
where we put degree $1$ for all variables $x_i$. 
We denote by $I_D \subset R$ the
Stanley-Reisner ideal of $D$ and we set $k[D]=R/I_D$.  
For $i=1, \dots , n+1$, we set 
\[
         f_i = \sum_{j=1}^m a_{i,j} x_j,
\]
and we define $A = k[D]/(f_1, \dots , f_{n+1})$.
Hence, $A$ is the generic Artinian reduction of $k_1[D]$ in the sense of 
Definition~\ref{dfn!genericartinianreduction}.
We denote by $ \pi :  R \to A$ the natural projection $k$-algebra
homomorphism. 

\begin{remark} \label{rem!basicpropertiesofA}
By \cite[Section~5]{BH},  the $k$-algebra $k[D]$ is  standard graded and Gorenstein
with Krull dimension equal to $n+1$. Since $a_{i,j}$ are independent variables
that do not appear in the minimal monomial generating set for $I_D$,
the sequence $f_1, \dots , f_{n+1}$ is a regular sequence for $k[D]$,
see \cite[Proposition~1.5.12]{BH}.  Hence,
$A$ is a Gorenstein Artinian standard graded $k$-algebra.
 It has socle degree equal to  $n+1$ by  \cite [Lemma~5.6.4]{BH}.  
Consequently,   $A_i = 0$  for all $i \geq n + 2$ and $\dim_k A_{n+1} =1$.
In particular,  $\dim_k A_{1}  \geq 1$, which implies that $m \geq n + 2$.  
\end {remark}

\begin{definition}   \label{defn!genericanisotropic}
We call $D$ generically anisotropic over $k_1$, 
if for all integers  $j$ with $1 \leq 2j \leq n+1$
and all nonzero elements 
$u \in A_ j$ we have $u^2 \not= 0$.
\end{definition}

The	 main result of the present paper is 
the following theorem, whose proof 
will be given in Subsection~\ref{subs!proofofmainthm}.

\begin{theorem} \label{thm!anisotropicityinchar2}
Assume that the field $k_1$ has characteristic $2$, $n \geq 1$ is
an integer,  and $D$ is a simplicial sphere
of dimension $n$.
Then  $D$ is generically anisotropic over $k_1$.
\end{theorem}

\section{The Artinian reduction of the Stanley-Reisner ring} \label{sec!artinioanreduction}

We keep using the notations and assumptions
defined  in Section~\ref{sec!statementofthemaintheorem}.
In particular, we allow the field $k_1$ to be of
arbitrary characteristic.

If $\sigma = (b_1,  \dots , b_q)$  is a sequence 
of integers, with  $1 \leq b_i \leq m$ for all $i$, we set 
\[
      x_{\sigma} = \prod_{i=1}^{q} x_i \in R.
\]
Whenever $q = n+1$,  we also use the notation
\[
           [\sigma] =  [ b_1,  \dots , b_{n+1}]  \in  k,
\]
where, by definition,   $[ b_1,  \dots , b_{n+1}]$ is the determinant
  of the  $(n+1) \times (n+1)$ matrix with $(i,j)$-entry equal
to  $a_{i, b_j}$.

We denote by $F(D)$ the set of facets of $D$.
We define an ordered facet of $D$ to be  a sequence $ (b_1, b_2,  \dots , b_{n+1})$
of positive integers such that the set $\{b_1, b_2,   \dots , b_{n+1}  \}$
is a facet of  $D$.
For $0 \leq i \leq n$,  we define a codimension $i$ face $\sigma$ of $D$ to 
be a face of dimension $n-i$. This is equivalent to $\; \# \sigma = n+1-i$.

Assume $g = \prod_{i=1}^m x_i^{a_i}  \in R$ is a monomial. We define 
the complexity $c(g)$ of $g$ by
\[
      c(g)  = \sum_{i=1}^m{a_i} -  \# \{i:  a_i  > 0  \}.
\]
It is clear that  $c(g) \geq 0$ and that 
 $c(g) = 0$ if and only if  $g$ is square-free.

The following proposition is well-known, but we provide a
proof for completeness.

\begin{proposition}
\label{prop!generationbysquarefree}
 Assume $1 \leq r \leq n+1$.  We have that the $r$-th graded piece
$A_r$ of $A$ is spanned, as a $k$-vector space, by the image under
$\pi$ of the set of  square-free monomials of $R$ of degree $r$.
\end{proposition}

\begin{proof} 
   By finite induction, it is enough to show that if $g \in R$ is
a nonzero monomial of degree $r$ and complexity $\geq 1$,
then there exists $q \in R$ homogeneous of degree $r$,
such that  $\pi(q) =\pi(g) $ and
$q$  is a linear combination  of monomials of complexity $c(g)-1$.

Assume $g = \prod_{i=1}^m x_i^{a_i}$.
Since $c(g) \geq 1$,  by rearranging indices
we can assume that $a_1 \geq 2$.  Since $ r \leq n+1$,
by rearranging indices we can assume that $a_i = 0$ for 
all $i \geq n + 2$.

By Proposition~\ref{prop!aboutGausselim},  we have
\[
        \sum_{t =1}^m [ 2,3, \dots , n+1,t] \pi (x_t) = 0.
\]
Hence,
\[
       [ 2,3, \dots , n+1,1] \pi (x_1) = 
        - \sum_{t =n+2}^m [ 2,3, \dots , n+1,t] \pi (x_t).
\]
As a consequence, multiplying by $\pi(g/x_1)$ we get
\[
    \pi(g)  =  -(\sum_{t =n+2}^m [ 2,3, \dots , n+1,t] \pi (x_tg/x_1))/ [ 2,3, \dots , n+1,1].
\]
Since, for all $t \geq n+2$, we have  $c(x_tg/x_1) = c(g)-1$, the result follows.
\end{proof}

\begin{remark}  \label{rem!seelaterforstrengthening}
   For a strengthening of Proposition~\ref{prop!generationbysquarefree}
  see Proposition~\ref{prop!dstrengtheningofsquarefree}.
\end{remark} 

\begin{remark} 
 \label{rem!abcdefg}
We will use  the following two facts,  see \cite[p.~111,
Remark before Corollary~7.19]{BP}.
Each codimension $1$ face of $D$ is
contained in exactly two facets of $D$. Moreover, if $\sigma_1$ and $\sigma_2$
are two facets of $D$, then there exists a finite sequence 
\[
     \tau_0,      \tau_1,  \dots ,      \tau_q  
\]
of facets of $D$ such that $\tau_0 = \sigma_1$,   $\tau_q = \sigma_2$,  
and, for all $0 \leq i \leq q-1$, the intersection
$\tau_i \cap \tau_{i+1}$ is a codimension $1$ face of $D$.
\end{remark}

\begin{proposition}
\label{prop!reducedEquation} 
Assume 
\[
     \sigma_1 = (b_1, \dots , b_{n}, d_1),  \quad    
                 \sigma_2 = (b_1, \dots , b_{n}, d_2), 
\]
are two ordered facets of $D$ having codimension $1$ intersection. We then have the following
equality in the ring $A$
\[
     [\sigma_1]  \pi(x_{\sigma_1})  =  -      [\sigma_2] \pi(x_{\sigma_2}).
\]
\end{proposition}

\begin{proof} 
   We set  $\tau = \sigma_1  \cap \sigma_2$. Hence, $\tau=\{  b_1, \dots , b_{n} \}$.
   By Proposition~\ref{prop!aboutGausselim}, we have that 
   \[
        \sum_{j=1}^m  [b_1, b_2, \dots , b_{n},j] \pi(x_j)   = 0.
   \]
   Hence,
  \[
        \sum_{j=1}^m  [b_1, b_2, \dots , b_{n}, j] \pi(x_jx_{\tau})   = 0. 
   \]
  If $j \in \tau$, we have $ [b_1, b_2, \dots , b_{n}, j] =0$.
   By Remark~\ref{rem!abcdefg},  $\sigma_1$ and $\sigma_2$ 
   are the only facets of $D$ which contain  the codimension $1$ face $ \tau$.
   Hence, the only terms of the last sum that are nonzero are for $j=d_1$ and  $j=d_2$.
   The result follows.
\end {proof}

\begin{corollary}
\label{cor!abouttwofacets} 
Assume  $\sigma_1$ and $\sigma_2$ are two ordered facets of $D$.  Then there exists
$\epsilon~\in~\{ -1, 1 \}$, such that
\[
     [\sigma_1] \pi(x_{\sigma_1})  = \epsilon    [\sigma_2] \pi(x_{\sigma_2}).
\]
\end{corollary}

\begin{proof} 
     By Remark~\ref{rem!abcdefg},  there exists a finite sequence 
\[
     \tau_0,      \tau_1,  \dots ,      \tau_q  
\]
of facets of $D$ such that $\tau_0 = \sigma_1$,   $\tau_q = \sigma_2$,  
and, for all $0 \leq i \leq q-1$, the intersection
$\tau_i \cap \tau_{i+1}$ is a codimension $1$ face of $D$.
Using Proposition~\ref{prop!reducedEquation}, we have that,
for all $0 \leq i \leq q-1$, there exists $\epsilon_i \in \{ -1, 1 \}$, 
such that  we have the following
equality in the ring $A$
\[
     [\tau_i] \pi(x_{\tau_i})  = \epsilon_i    [\tau_{i+1}] \pi(x_{\tau_{i+1}}).
\]
The result follows.
\end{proof}

We fix an ordered facet $e=(e_1, \dots , e_{n+1})$ of $D$.
By Remark~\ref{rem!basicpropertiesofA},  $\dim_k A_{n+1} =1$.
Using Proposition~\ref{prop!generationbysquarefree}, $A_{n+1}$
is spanned, as a $k$-vector space,  
by the square-free monomials that correspond to 
the facets of $D$.  
Corollary~\ref{cor!abouttwofacets} implies that any of them spans
 $A_{n+1}$.
As a consequence,   $\pi(x_e) \not= 0$ 
and  $\pi(x_e)$ is a $k$-basis of $A_{n+1}$.
Hence, there exists a  unique  set-theoretic map $\Psi_e : A_{n+1} \to k$
with the property that
\begin{equation} \label{eqn!dfnofcapitalpsi}
    u =  \Psi_e (u) [e] \pi (x_e)   
\end {equation}
for all $u \in  A_{n+1}$.  
It is clear that $\Psi_e$ is an isomorphism of $k$-vector spaces.
In addition, if $n$ is odd, we set $p_1 = (n+1)/2$ and
define the symmetric  bilinear form 	
\begin{equation} \label{eqn!dfnofrhoe}
        \rho_e : A_{p_1} \times A_{p_1} \to k
\end {equation}
by
\[
     \rho_e (u,w) =  \Psi_e (uw)
\]
for all $u,w \in A_{p_1}$.

\begin{remark} 
 \label{rem!signofPsi}
   If we change the ordered facet $e$ of $D$ to another
   ordered facet $\sigma$, Corollary~\ref{cor!abouttwofacets} implies
   that either $ \Psi_{\sigma} =  \Psi_{e}$ or  $ \Psi_{\sigma} =  -\Psi_{e}$.
   Hence, if the field $k_1$ has characteristic $2$ the map $\Psi_e$
   is canonical, in the sense that it is independent of the choice of the 
   facet $e$ of $D$, and we will denote it by $\Psi$. 
\end{remark}

\begin{remark} 
 \label{rem!relnbwithanisotropicity}
   Assume $n$ is odd.  Recall that a symmetric bilinear form 
    $\delta : A_ {p_1} \times A_ {p_1}  \to k$ is called anisotropic 
   if   $\delta (u,u) \not= 0$ for all  nonzero elements  $u \in A_ {p_1}$.
   Using Remark~\ref{rem!poincare_duality_for_gor},
   it follows  that $\rho_e$ is anisotropic if and only if  
    for all integers  $j$ with $1 \leq 2j \leq n+1$
   and all nonzero elements  $u \in A_ j$ we have $u^2 \not= 0$.   
   This (partially) explains the use of  the term generic anisotropicity
   in Definition~\ref{defn!genericanisotropic}.
\end{remark}

\begin{remark} 
 \label{rem!imageofPsi}
    The proof of Proposition~\ref{prop!generationbysquarefree}       
    gives that for all $u \in  A_{n+1}$ the 
    element  $\Psi_e (u)$ of $k$  is a rational function in the
     set of all bracket polynomials 
  \[
           \{ \; [i_1, \dots , i_{n+1}] \; : \; 1 \leq i_1 < i_2 < \dots < i_{n+1} \leq m \; \}.
   \]
   In addition,  combined with the proof of
   Corollary~\ref{cor!abouttwofacets},      
    it provides an algorithm for computing   $\Psi_e (u)$.  
\end{remark}

\begin{proposition}
\label{prop!squarefreeinchar2}
Assume $k_1$ is a field of characteristic $2$ and
$\sigma = ( b_1, \dots , b_{n+1} )$ is a  facet of $D$.
We have 
\[ 
        (\Psi  \circ   \pi)   (x_{\sigma})   = 1 / 
                 [b_1, \dots , b_{n+1}].
\]
\end {proposition}

\begin{proof} 
 By Corollary~\ref{cor!abouttwofacets}, we have 
\[
     [\sigma] \pi(x_{\sigma})  =    [e] \pi(x_{e}).
\]
The result follows from the definition of $\Psi$.           
\end{proof}

The following proposition
allows the computation of $\Psi_e (u)$ in more cases.

\begin{proposition}
\label{prop!oneSquareequation} 
Assume  $\sigma = (b_1, \dots , b_{n-1}, c )$ is a codimension $1$ ordered face of $D$.
Denote by
 $\; \tau_1 = (b_1, \dots , b_{n-1}, c, d_1) \;$  and  $\; \tau_2 = (b_1, \dots , b_{n-1}, c, d_2) \;$
 the two ordered facets of $D$ that contain $\sigma$.
We then have the following  two equalities
\begin{eqnarray*}
     [b_1, \dots , b_{n-1}, c, d_1]    [b_1, \dots , b_{n-1}, c, d_2] 
         \pi(x_c^2 \prod_{i=1}^{n-1} x_{b_i})  & = &
                   -[b_1, \dots , b_{n-1}, d_1, d_2]    [\tau_1] \pi(x_{\tau_1})  \\
    &  = &   \phantom{-}[b_1, \dots , b_{n-1}, d_1, d_2]   
                            [\tau_2] \pi(x_{\tau_2}).
\end{eqnarray*}
\end{proposition}

\begin{proof}     We set $S = \{1, \dots , m\} \setminus \{ c \}$.
   By Proposition~\ref{prop!aboutGausselim}, we have that 
   \[
        \sum_{j=1}^m  [b_1, b_2, \dots , b_{n-1}, d_1, j] \pi(x_j)   = 0.
   \]
   Hence,
   \[
          [b_1, b_2, \dots , b_{n-1}, d_1, c] \pi(x_c) =    
                  -  \sum_{j \in S}  [b_1, b_2, \dots , b_{n-1},c,j]\pi( x_j) . 
   \]
   Consequently,
   \[
          [b_1, b_2, \dots , b_{n-1}, d_1, c]  \pi(x_c^2 \prod_{i=1}^{n-1} x_{b_i})
                   =         - \sum_{j \in S}  [b_1, b_2, \dots , b_{n-1},d_1,j] 
                                    \pi(x_j x_c \prod_{i=1}^{n-1} x_{b_i}). 
   \]

   Arguing for the  last  sum as in the proof of  Proposition~\ref{prop!reducedEquation}, we get
   \[
          [b_1, b_2, \dots , b_{n-1}, d_1, c]  \pi(x_c^2 \prod_{i=1}^{n-1} x_{b_i})
                   =         -  [b_1, b_2, \dots , b_{n-1},d_1,d_2] 
                          \pi(x_{d_2} x_c \prod_{i=1}^{n-1} x_{b_i}). 
   \]
    Using that, by Proposition~\ref{prop!reducedEquation},    
           $\; [\tau_1] \pi(x_{\tau_1})  =  -      [\tau_2] \pi(x_{\tau_2})$,
    the result follows.
\end{proof}

The following corollary is an immediate consequence of
Proposition~\ref{prop!oneSquareequation}.

\begin{corollary}
\label{cor!oneSquareequationcorollary} 
Assume $k_1$ is a field of characteristic $2$
and  $\sigma = (b_1, \dots , b_{n-1}, c )$ is a codimension $1$ face of $D$.
Denote by
 $(b_1, \dots , b_{n-1}, c, d_1)$ and  
$(b_1, \dots , b_{n-1}, c, d_2)$ the two facets of $D$ that contain $\sigma$.
We have 
\[ 
      (\Psi  \circ   \pi)  (x_c^2 \prod_{i=1}^{n-1} x_{b_i})  = \frac {
              [b_1, \dots , b_{n-1}, d_1, d_2]} {
                 [b_1, \dots , b_{n-1}, c, d_1]    [b_1, \dots , b_{n-1}, c, d_2] }.
\]
\end{corollary}

\vspace{8pt}

{\bf FURTHER ASSUMPTION.}
For the rest of this section we make the additional
assumption that the field $k_1$
has characteristic $2$. 

 We set  $Z = m + 2n$  and 
denote by  $M$ the $(n+1) \times Z$ matrix whose
$(i,j)$-entry is equal to the variable $a_{i,j}$, for
$1 \leq i  \leq n+1$ and  $1 \leq j  \leq Z$.  Given 
a subset $\mathcal{A}$ of the set $\{1,2, \dots , Z\}$
of cardinality $n+1$, we denote by  $M(\mathcal{A})$
the determinant of the $(n+1) \times (n+1) $ submatrix
of $M$ obtained by keeping the columns of $M$
specified by the set   $\mathcal{A}$.

We denote by $k_2$ the field of fractions of the polynomial
ring
\[
             k_1 [ a_{i,j} :  1 \leq i \leq n+1,  \;  1 \leq j \leq Z ].
\]
It follows that $k$ is a subfield of $k_2$.

\begin{proposition}
\label{prop!simplexnoddformula} 
 (Recall that the field $k_1$ has characteristic equal to $2$.)
Assume $n$ is odd. We set $l= (n+1)/2$.  Assume $D$ is
the boundary complex of the $(n+1)$-dimensional simplex 
with vertex set  $ \tau = \{c_1, \dots , c_l, g_1,  \dots  ,g_{l+1} \}$.
We then have the following equality in the field~$k_2$
\begin{eqnarray*}
     (\Psi  \circ   \pi)   (\prod_{i=1}^{l} x_{c_i}^2)  = \frac {
                        \prod_{i=1}^{l}  M(\tau \setminus \{ c_i \})  }  { 
                        \prod_{i=1}^{l+1}  M( \tau \setminus \{ g_i \}) }.
\end{eqnarray*}

\end{proposition}

\begin{proof}    We set $c= \{c_1, \dots , c_l  \}, g = \{ g_1,  \dots  ,g_{l+1} \}$.
Assume $1 \leq i \leq l$.
By Proposition~\ref{prop!aboutGausselim}, we have that 
   \[
        \sum_{t=1}^{l} [ c \setminus \{c_i\},  g \setminus \{g_i\} , c_t] \pi(x_{c_t})  
         +  \sum_{t=1}^{l+1} [ c \setminus \{c_i\},  g \setminus \{g_i\} , g_{t} ] \pi(x_{g_t})  = 0.
   \]
Hence,
   \[
       [ c \setminus \{c_i\},  g \setminus \{g_i\} , c_i] \pi(x_{c_i})   =  [ c \setminus \{c_i\},  g \setminus \{g_i\} , g_i] \pi(x_{g_i}),
   \]
since the field has characteristic $2$ and the other terms in the two sums are zero.

Multiplying the above equations for $1 \leq i \leq l$, we get
\begin{equation}  \label{eqn!sersertu}
       (  \prod_{i=1}^{l} [ c \setminus \{c_i\},  g \setminus \{g_i\} , c_i]  )  \; u_1  = 
              (  \prod_{i=1}^{l} [ c \setminus \{c_i\},  g \setminus \{g_i\} , g_i] ) \; u_2,
\end{equation} 
where
\[
     u_1  =  \prod_{i=1}^{l} \pi(x_{c_i}) , \quad  \quad  u_2  =  \prod_{i=1}^{l} \pi(x_{g_i}).
\]  
The result follows by multiplying both sides of Equality~(\ref{eqn!sersertu})   by  $u_1$ and using that, by Corollary~\ref{prop!squarefreeinchar2},
\[
      \Psi   (  u_1 u_2 )  =  1/[ c ,  g \setminus \{g_{l+1}\}]. 
\]
\end{proof}

\begin{proposition}
\label{prop!simplexnevenformula} 
 (Recall that the field $k_1$ has characteristic equal to $2$.)
Assume $n$ is even. We set $l= n/2$.  Assume $D$ is
the boundary complex of the $(n+1)$-dimensional simplex 
with vertex set  $\tau = \{   c_1, \dots , c_l,  b,    g_1,  \dots  ,g_{l+1} \}$.
We then have the following equality in the field $k_2$
\begin{eqnarray*}
     (\Psi  \circ   \pi)   (x_b \prod_{i=1}^{l} x_{c_i}^2 )  = \frac {
                        \prod_{i=1}^{l}  M( \tau \setminus \{ c_i \})  }  { 
                        \prod_{i=1}^{l+1}  M( \tau \setminus \{ g_i \}) }.
\end{eqnarray*}

\end{proposition}

\begin{proof}    We set $c= \{c_1, \dots , c_l  \}, g = \{ g_1,  \dots  ,g_{l+1} \}$.
Assume $1 \leq i \leq l$.
By Proposition~\ref{prop!aboutGausselim}, we have that 
   \[
        \sum_{t=1}^{l} [ b, c \setminus \{c_i\},  g \setminus \{g_i\} , c_t] \pi(x_{c_t})  
         +  \sum_{t=1}^{l+1} [b,  c \setminus \{c_i\},  g \setminus \{g_i\} , g_{t} ] \pi(x_{g_t})  = 0.
   \]
Hence,
   \[
       [ b, c \setminus \{c_i\},  g \setminus \{g_i\} , c_i] \pi(x_{c_i})   =  [b,  c \setminus \{c_i\},  g \setminus \{g_i\} , g_i] \pi(x_{g_i}),
   \]
since the field has characteristic $2$ and the other terms in the two  sums are zero.

Multiplying the above equalities for $1 \leq i \leq l$, we get
\begin{equation}  \label{eqn!efdefdefd}
           (\prod_{i=1}^{l} [ b, c \setminus \{c_i\},  g \setminus \{g_i\} , c_i]) \; u_1    = 
               ( \prod_{i=1}^{l} [ b, c \setminus \{c_i\},  g \setminus \{g_i\} , g_i]) \; u_2 ,
\end{equation} 
where
\[
     u_1  =  \prod_{i=1}^{l} \pi(x_{c_i}) , \quad  \quad  u_2  =  \prod_{i=1}^{l} \pi(x_{g_i}).
\]  
The result follows by multiplying  both sides of Equality~(\ref{eqn!efdefdefd})   
by  $ \; \pi(x_b)u_1 \; $ and using that, by Corollary~\ref{prop!squarefreeinchar2},
\[
     \Psi  ( \pi(x_b) u_1 u_2 )  =  1/[ b, c ,  g \setminus \{g_{l+1}\}]. 
\]
\end{proof}

We fix an integer $r$ with $m + 1 \leq r \leq Z$.
Assume $l$ is an integer with $ 2 \leq  2l  \leq n+1$.
We set $s= n+1 - 2l$.   Assume
\[
    \tau_1 = \{ c_1,  \dots , c_{l} \} , \quad   \tau_2 = \{ b_1,  \dots , b_{s} \}
\]
are two subsets of the vertex set  $\{1, \dots , m  \}$ 
of $D$, such that $\tau_1 \cup \tau_2$ has
cardinality $l+s$ and is a face of $D$.  We set  $\tau =\tau_1 \cup \tau_2$.

\noindent Assume $\sigma \in F(D)$ is a facet of $D$.
We define the rational function $H( \tau_1, \tau_2, \sigma)$ as follows:

1. If  $\tau$ is not a subset of $\sigma$ we set  $H( \tau_1, \tau_2, \sigma) = 0$.

2. If   $\tau$ is  a subset of $\sigma$,  we denote the elements of  
             $\sigma  \setminus \tau$ by $g_1, \dots , g_l$ and  we set  
\[ 
        H  ( \tau_1, \tau_2, \sigma) =    \frac {
                        \prod_{i=1}^{l}  M(  (\sigma \cup \{ r \} )  \setminus \{ c_i \})  }  { 
                           M(\sigma)  \prod_{i=1}^{l}  M( (\sigma \cup \{ r \} ) \setminus \{ g_i \}) }.
\]
Clearly,
\[ 
        H  ( \tau_1, \tau_2, \sigma) =    \frac {
                        \prod_{j \in \tau_1}   M(  (\sigma \cup \{ r \} )  \setminus \{ j \})  }  { 
                           M(\sigma)  \prod_{j \in (\sigma \setminus (\tau_1 \cup \tau_2))}  M( (\sigma \cup \{ r \} ) \setminus \{ j \}) }.
\]

The proof of the following theorem will be given in 
Subsection~\ref{subsect!proofthmlsquares}.

\begin{theorem}
\label{thm!aboutlsquares}
 (Recall that the field $k_1$ has characteristic equal to $2$.)
We have the following equality in the field $k_2$
\begin{equation}  \label{eqn!psipiintermsof}
     (\Psi  \circ   \pi)   ( (\prod_{i=1}^{l} x_{c_i}^2) (\prod_{i=1}^{s} x_{b_i}) )  = \sum_{\sigma \in F(D)}   H( \tau_1, \tau_2, \sigma).
\end{equation} 
\end{theorem}

\begin{remark}  It is interesting to notice the similarities in
the statement  and proof of Theorem~\ref{thm!aboutlsquares} with the
results obtained by Lee in \cite[Section~6]{Lee}.
\end{remark}

\begin{remark}  Using the definition of the function $H$, it is clear
that the nonzero terms of the sum in Equation~(\ref{eqn!psipiintermsof})
are exactly those where $\sigma$  contains $\tau$.  Hence, the sum 
can also be considered as  a sum over the facets of the link 
(\cite[Definition~5.3.4]{BH}) of the face $\tau$ in $D$. 
\end{remark}

\begin{remark} Even though the left hand side in
Equation~(\ref{eqn!psipiintermsof}) is completely independent of $r$, each 
nonzero term $H( \tau_1, \tau_2, \sigma)$  on the right hand side 
does depend on $r$.  Hence, provided no denominator vanishes,
we are allowed  
to specialise the variables $a_{i,r}$,  for $1 \leq i \leq n+1$.
This observation
will be used  in Corollaries~\ref{cor!sumofGafornoddwere}    
and~\ref{cor!sumofGafornevenwere}.
\end{remark}

\begin {example}  \label{example!about1dimcaseno1}
 Assume  $k_1$ is a field
of characteristic $2$,  $m \geq 3$ and $D$ is the $m$-gon with
consecutive vertices $1,2, \dots , m$.   By Corollary~\ref{cor!oneSquareequationcorollary}, 
we have
\[
     (\Psi  \circ   \pi)   (  x_{2}^2)   =      \frac {
              [1,3]} {[1,2]    [2,3] },
\]
while, by  Theorem~\ref{thm!aboutlsquares},   we have
\[
     (\Psi  \circ   \pi)   (  x_{2}^2)   =    H (\{ 2 \},\newemptyset,\{1,2\} ) + H (\{ 2 \},\newemptyset,  \{2,3 \}) =
           \frac {
              [1, r]} {[1,2]    [2,r] } +      \frac {
              [3,r]} {[2,3]    [2,r] }.
\]
\end {example}

\begin {example}  \label{example!a3dimensionalexample}
 Assume  $k_1$ is a field
of characteristic $2$,  and $D$ is the simplicial complex 
with vertex set $\{1,2, \dots , 7 \}$ 
and Stanley-Reisner ideal equal to 
$I_D = (x_1x_2,x_3x_4x_5, x_6x_7)$. Then $D$
is a simplicial sphere of dimension $3$.

We set $\tau_1=\{1,3\}, \tau_2=\newemptyset$.
Clearly $\tau_1$ is a face of $D$.
Since \ $I_D : (x_1x_3)= (x_2, x_4x_5, x_6x_7)$, the link
of $\tau_1$ in $D$ is the $4$-gon with consecutive
vertices $4,6,5,7$. 
By Theorem~\ref{thm!aboutlsquares}
\[
     (\Psi  \circ   \pi)   (x_{1}^2x_{3}^2)   = 
          H_{4,6}+H_{6,5} +H_{5,7} + H_{7,4},
\]
where 
\[
  H_{a,b} =  H ( \tau_1, \tau_2,  \tau_1 \cup \{ a, b \}) =
     \frac { [1, a, b, r][3, a, b, r]  } 
     {[1, 3, a, b ][1, 3, a, r] [1, 3, b, r] }.
\]

\end {example}

\begin{remark}  \label{rem!aboutcorrectingsigns}
 We expect that with the correct sign
adjustments  there should be a version of 
Theorem~\ref{thm!aboutlsquares}  valid over a field  $k_1$ 
of arbitrary characteristic.  We do not pursue this direction
further in the present  work.
\end{remark}

\subsection{The proof of Theorem~\ref{thm!aboutlsquares}}
   \label{subsect!proofthmlsquares} 

We now give the proof of Theorem~\ref{thm!aboutlsquares} by
induction on $l \geq 1$.

Assume $l=1$. 
We have  $s=  n-1$ and
\[
    \tau_1 = \{ c_1 \} , \quad   \tau_2 = \{ b_1,  \dots , b_{n-1} \}.
\]
Recall that $\tau = \tau_1 \cup \tau_2$. Hence, $\tau$ is a codimension $1$ face of $D$.
Using Remark~ \ref{rem!abcdefg}, $\tau$
 it is contained in exactly two facets of $D$. 
We denote them by 
 \[
         \sigma_1 = \{ b_1, \dots , b_{n-1}, c_1, d_1 \}, \quad    \sigma_2 = \{ b_1, \dots , b_{n-1}, c_1, d_2 \}.
\]
We use the notation
\[
       [ \tau_2, i,j ]  =  [b_1, \dots , b_{n-1},  i,j].
\]
By  Corollary~\ref{cor!oneSquareequationcorollary},
\[ 
      (\Psi  \circ   \pi)  (x_{c_1}^2 \prod_{i=1}^{n-1} x_{b_i})  = \frac {
              [\tau_2, d_1, d_2]} {
                 [\tau_2, c_1, d_1]    [\tau_2, c_1, d_2] }.
\]
We have  $H( \tau_1,  \tau_2, \sigma) = 0$ if
$\sigma \in  F(D) \setminus \{ \sigma_1, \sigma_2  \}$.
Using the Pl\"ucker relation (\cite[Theorem~5.2.3]{LB})
\[
      [\tau_2, d_1, d_2]   [\tau_2, c_1, r] = 
      [\tau_2, d_1, r]   [\tau_2, c_1, d_2] + [\tau_2, d_1, c_1]   [\tau_2, d_2, r]
\]
and taking into account that the field $k_1$ has charactersitic $2$,  we have
\begin {align*}
    (\Psi  \circ   \pi)  (x_{c_1}^2 \prod_{i=1}^{n-1} x_{b_i})  & =  \frac {  [\tau_2, d_1, d_2]} {[\tau_2, c_1, d_1]    [\tau_2, c_1, d_2] }  \\
        & =   \frac { [\tau_2, d_1, d_2]   [\tau_2, c_1, r] }  { [\tau_2, c_1, d_1]    [\tau_2, c_1, d_2]    [\tau_2, c_1, r]}   \\
       & =    \frac { [\tau_2, d_1, r]   [\tau_2, c_1, d_2] + [\tau_2, d_1, c_1]   [\tau_2, d_2, r]  } 
                         { [\tau_2, c_1, d_1]    [\tau_2, c_1, d_2]    [\tau_2, c_1, r]}   \\
    &  =         \frac {  [\tau_2, d_1, r]} {         [\tau_2, c_1, d_1]    [\tau_2, c_1, r] }    +    \frac {    [\tau_2, d_2, r]} {
                                    [\tau_2, c_1, d_2]   [\tau_2, c_1, r]   }  \\
    &  = \;        H( \tau_1,  \tau_2, \sigma_1) +  H( \tau_1,  \tau_2, \sigma_2). 
\end {align*}

\vspace{18pt}
We assume now that $l \geq 1$ with  $2(l+1) \leq n+1$ and that
Theorem~\ref{thm!aboutlsquares} is true for $l$. We will
prove that Theorem~\ref{thm!aboutlsquares} is true for the
value $l+1$. We set $s= n+1 - 2(l+1)$.   Assume
\[
    \tau_1 = \{ c_1,  \dots , c_{l+1} \} , \quad   \tau_2 = \{ b_1,  \dots , b_{s} \}
\] 
such  that  $\tau_1 \cup \tau_2$ has
cardinality $l+s +1$ and is a face of $D$.

We fix integers $p_1,  \dots ,p_{l}$,  such that
$m+1 \leq p_i \leq Z$, for all $1 \leq i \leq l$,  and
the set $\{r, p_1, p_2, \dots , p_{l}\}$
has cardinality equal to $l+1$.  
We set $\; \mathcal{B} =  \{ 1 , \dots ,m \} \setminus  (\tau_1 \cup \tau_2)\;$
and
\[
      u = \big( \prod_{i=1}^{l+1} x_{c_i}^2  \big) \big(  \prod_{i=1}^{s} x_{b_i} \big).
\]
For $1 \leq i \leq Z$,  we set
\[
            [[ i ]] = [ b_1, \dots ,b_s, \;  c_1,  \dots ,c_l, \;  r,  \; p_1,  \dots ,p_l ,   \;   i].
\]
Using Proposition~\ref{prop!aboutGausselim},
\[
        \sum_{i=1}^{m}    [[ i ]]  \; \pi (x_i) = 0.
\]
Hence
\[
            \pi (x_{c_{l+1}})=     \sum_{i} \frac  {  [[ i ]]    }
                      {  [[ c_{l+1}  ]]   } \;  \pi (x_i),
\]
with the sum for  $1 \leq i \leq m$ and $i \not= c_{l+1}$.  Since  $ [[i]] =0$ when
$i \in \{c_1, \dots , c_l, b_1, \dots , b_s \}$,  we have that 
\[
                    \pi (x_{c_{l+1}})=     \sum_{i \in \mathcal{B} } \frac  {  [[i]] }  { [[   c_{l+1}  ]] } \;  \pi (x_i).
\]
Multiplying this equality by 
\[
        \pi (x_{c_{l+1}})    \prod_{i=1}^{l} \pi (x_{c_i}^2)  \prod_{i=1}^{s} \pi (x_{b_i})  
\]
we get
\[
      \pi    (u) =   
                  \sum_{i \in \mathcal{B} } \frac  {  [[i]] }  { [[   c_{l+1}  ]] }  \;  \pi (E_i),
\]
where 
\[    
        E_i =    x_i x_{c_{l+1}}  ( \prod_{i=1}^{l}  x_{c_i}^2 )  (  \prod_{i=1}^{s} x_{b_i}).
\]
Hence, 
\[
    ( \Psi \circ  \pi) (u)  =   
                  \sum_{i \in \mathcal{B} } \frac  {  [[i]] }  { [[   c_{l+1} ]] }  \; ( \Psi \circ  \pi)  (E_i).
\]

Since, for all $i \in \mathcal{B}$,  the expression for  $E_i$ has $l$ squares,  
we  can use the inductive hypothesis for  $( \Psi \circ  \pi)  (E_i)$  to get
\[
     ( \Psi \circ  \pi)  (E_i) =     
          \sum_{\sigma \in F(D) }  H (\tau_1 \setminus \{ c_{l+1}\}, \tau_2 \cup \{ i ,  c_{l+1} \}, \sigma).
\]  
As a consequence,
\[
    ( \Psi \circ  \pi) (u)  =   
                  \sum_{i \in \mathcal{B} }   \sum_{\sigma \in F(D) }  V_{i,\sigma}  =
               \sum_{\sigma \in F(D) }     \sum_{i \in \mathcal{B} }   V_{i,\sigma},
\]
where 
\[
              V_{i,\sigma}  =  \frac  {  [[i]] }  { [[ c_{l+1} ]] }    \; 
                        H (\tau_1 \setminus \{ c_{l+1}\}, \tau_2 \cup \{ i ,  c_{l+1} \},
 \sigma) .
\]
Therefore, to finish the proof it is enough to show that  for all  $\sigma \in F(D)$ it holds
\begin{equation}  \label {eqn!eqnforVinthetheoremproof}
                \sum_{i \in \mathcal{B} }   V_{i,\sigma} =  H (\tau_1, \tau_2, \sigma).
\end{equation} 

For $i \in \mathcal{B}$ we set  
\[
       \eta_i=  (\tau_1 \setminus \{ c_{l+1}\}) \cup  (\tau_2 \cup \{ i ,  c_{l+1} \} ),
\]
therefore  $\eta_i  =  \tau \cup \{ i \}$.

We first assume that  $\sigma \in F(D)$ does not contain  $ \tau$ as a subset.  
Hence $ H (\tau_1, \tau_2, \sigma) =0$.    Assume   $i \in \mathcal{B}$. 
Since  $\tau \subset \eta_i$, it follows that $\eta_i$  is not a subset of $\sigma$. This implies that 
$H (\tau_1 \setminus \{ c_{l+1}\}, \tau_2 \cup \{ i ,  c_{l+1} \})= 0$, therefore
$V_{i,\sigma} = 0$. As a consequence, Equality~(\ref {eqn!eqnforVinthetheoremproof}) is true.

Assume now that $\sigma \in F(D)$ contains $\tau$ as a subset.  
We set $\; \mathcal{C} = \sigma \setminus \tau$ and 
denote the elements of $\mathcal{C}$   by $g_1, \dots , g_{l+1}$.
We set $\sigma^{r} =\sigma \cup \{ r \}$.
If $i \in \mathcal{B} \setminus \mathcal{C}$, it follows that $\eta_i$  is not a subset of $\sigma$, therefore
$V_{i,\sigma} = 0$. As a consequence,
\[
     \sum_{i \in \mathcal{B} }   V_{i,\sigma}  =    \sum_{i \in \mathcal{C} }   V_{i,\sigma} 
            = \sum_{i=1}^{l+1}  V_{g_i,\sigma}.
\]
 We have
\begin {align*}
              V_{g_i,\sigma}  & =    \frac  {[  [ g_i  ]] }  { [[   c_{l+1}  ]] }   \; 
                        H (\tau_1 \setminus \{ c_{l+1}\}, \tau_2 \cup \{ g_i ,  c_{l+1} \}, \sigma)    \\
                       & =    \frac  {  [[ g_i  ]] }  { [[   c_{l+1}  ]] }   \; \frac {  
                                        \prod_{t=1}^{l}   M(  \sigma^{r}  \setminus \{ c_t \})  
                                              } {  
                                  M(\sigma)  \prod_{t=1}^{i-1}  M( \sigma^{r}  \setminus \{ g_t \})  \prod_{t=i+1}^{l+1}  M( \sigma^{r}  \setminus \{ g_t \})     }   \\
                       &  =    \frac  {  [[ g_i  ]] }  { [[   c_{l+1}  ]] }   \; \frac { 
                          M(  \sigma^{r} \setminus \{ g_i \})  \prod_{t=1}^{l}   M(  \sigma^{r}  \setminus \{ c_t \})  
                                              } { 
                           M(\sigma)  \prod_{t=1}^{l+1}  M( \sigma^{r}  \setminus \{ g_t \}) }    \\
             &=  \;  \Gamma      \;   [[ g_i  ]]  \; M( \sigma^{r}  \setminus \{ g_i \}),        
\end {align*}
where
\[
  \Gamma  =  \frac  {  \prod_{t=1}^{l}  M(  \sigma^{r}   \setminus \{ c_t \}) } { 
                                      [[   c_{l+1}  ]]    M(\sigma)  \prod_{t=1}^{l+1}  M( \sigma^{r}  \setminus \{ g_t \}) }.
\] 	
Hence,
\[
    \sum_{i=1}^{l +1}  V_{g_i,\sigma} = \Gamma  \sum_{i=1}^{l +1}   [[ g_i  ]]  M( \sigma^{r}  \setminus \{ g_i \}).
\]                 
By the Pl\"ucker relation (\cite[Theorem~5.2.3]{LB}),
\[
     \sum_{i=1}^{l+1}    [[ g_i  ]]  M( \sigma^{r}  \setminus \{ g_i \}) 
                    =   [[   c_{l+1}  ]]  M( \sigma^{r}  \setminus \{  c_{l+1} \}).
\]
Therefore,
\begin {align*}
     \sum_{i=1}^{l +1}  V_{g_i,\sigma} & = \Gamma \; [[   c_{l+1}  ]]  \; M ( \sigma^{r}  \setminus \{  c_{l+1} \})  \\
                  &   = \frac  {  \prod_{t=1}^{l+1}  M(\sigma^{r}   \setminus \{ c_t \}) } { 
                                      M(\sigma)  \prod_{t=1}^{l+1}  M( \sigma^{r}  \setminus \{ g_t \}) }  \\
                 & = \; H (\tau_1, \tau_2, \sigma).
\end {align*}
As a consequence, Equality~(\ref {eqn!eqnforVinthetheoremproof}) is true, which finishes the proof of 
Theorem~\ref{thm!aboutlsquares}.

\section{Using the differential operators to establish anisotropicity} \label{sec!usingthefdifferentialoperatorsto}

We keep using the notations introduced in 
Sections~ \ref{sec!statementofthemaintheorem} and  \ref{sec!artinioanreduction}.
Moreover, we assume that the field $k_1$ has characteristic $2$.  

\subsection {Case $n$ is odd} Assume $n \geq 1$ is odd. We set $l=(n+1)/2$.
We assume $\sigma \in D$ is a face  of dimension $l-1$. 
We denote, in increasing order,  the elements of $\sigma$  
by $\sigma(1),\sigma(2), \dots , \sigma(l)$. We define 
$\partial_{\sigma} : k_2 \to k_2 $ to be the $(n+1)$-th order differential operator
which is differentiation with respect to the variables in the set
\[
                \{ a_{i, \sigma(j)} \;  \;: \;   \; 1 \leq i \leq n+1,    \;  j = [ (i+1)/2 ]        \},
\]
where $[x]$  denotes the integral part of the real number $x$.

\begin{proposition}  \label{prop!differoddnpfofmainthm} 
Assume $\tau$ is a face of $D$ of dimension $l-1$.  
We then have  
\begin{equation}   \label{eqn!relatingpartialxtausquarewith}
          (\partial_{\sigma} \circ \Psi \circ \pi )   (x_{\tau}^2) =
              \big( (  \Psi \circ \pi )   (x_{\sigma}x_{\tau}) \big)^2.
\end{equation}
\end{proposition}

\begin{proof}   We define the  sets
\[
     \mathcal{K}_1 = \{\eta \in F(D) : \tau  \subset \eta  \}, \quad \quad 
     \mathcal{K}_2 = \{\eta \in F(D) : \tau \cup \sigma  \subset \eta  \}.
\]
We set $\gamma_1 = \tau \cap \sigma$, 
$\gamma_2 = (\tau \cup \sigma) \setminus \gamma_1$.
Using Theorem~\ref{thm!aboutlsquares}, we get
\[
     (\Psi  \circ   \pi)  (x_\tau^2)   =
              \sum_{ \eta \in  \mathcal{K}_1 }   H( \tau, \newemptyset , \eta) \quad
   \quad   \text{and}  \quad   \quad 
     (\Psi  \circ   \pi)  (x_{\tau} x_{\sigma})   =
              \sum_{ \eta \in  \mathcal{K}_2 }   H( \gamma_1,  \gamma_2 , \eta).
\]
Clearly $\mathcal{K}_2 \subset \mathcal{K}_1$.
If $\eta \in   \mathcal{K}_1 \setminus \mathcal{K}_2$, 
we have that $ \sigma \setminus  (\eta \cap \sigma )  \not= \newemptyset$, which implies that
$ \partial_{\sigma} ( H( \tau, \newemptyset , \eta)) = 0$.  Hence
\[
     (\partial_{\sigma} \circ \Psi  \circ   \pi)  (x_\tau^2)   =
      \sum_{ \eta \in  \mathcal{K}_2 }   \partial_{\sigma} (H( \tau, \newemptyset , \eta)).
\]
Since the field $k_1$ has characteristic $2$, we get
\[
    ( (\Psi  \circ   \pi)  (x_{\tau} x_{\sigma}))^2   =
              \sum_{ \eta \in  \mathcal{K}_2 }   (H( \gamma_1,  \gamma_2 , \eta))^2.
\]

Assume that $\eta \in \mathcal{K}_2$.  Using   Corollary~\ref{cor!oddfgshffosdfs}, we have
\begin{align*}
     \partial_{\sigma} (H( \tau, \newemptyset , \eta))   & =     \partial_{\sigma} \big(
            \frac {  \prod_{i \in \tau}   M(  (\eta \cup \{ r \} )  \setminus \{  i \})   }  { 
           M(\eta) \cdot \prod_{i \in \eta \setminus \tau} 
              M( ( \eta \cup \{ r \} ) \setminus \{ i \}) } \big)  \\
      & =   \frac {   \prod_{i \in \tau \cap \sigma }  (M(  (\eta \cup \{ r \})  \setminus \{  i \}))^2  }  { 
                                      (M(\eta))^2  \cdot  \prod_{i \in \eta \setminus (\tau \cup \sigma)} (M( ( \eta \cup \{ r \} ) \setminus \{ i \}))^2 } \\
     &=    (H( \gamma_1,  \gamma_2 , \eta))^2,
\end{align*}
which finishes the proof.   \end{proof}

\begin{remark}  Conjecture~\ref{conj!generalisingthmcomputingdifferoperators} 
contains  a conjectural statement generalising Proposition~\ref{prop!differoddnpfofmainthm}.
\end{remark}

\begin{corollary}  \label{cor!differoddnpfofmainthm} 
Assume $u$ is a homogeneous element of $R$ of degree $l$. 
We then have  
\begin{equation}  \label{eqn!relatingpartialusquarewith}
          (\partial_{\sigma} \circ \Psi \circ \pi )   (u^2) =
              \big( (  \Psi \circ \pi )   (x_{\sigma}u) \big)^2.
\end{equation}
\end{corollary}

\begin{proof}  Using  Proposition~\ref{prop!generationbysquarefree},
there exist  $s > 0$,  faces $\tau_1, \dots , \tau_s$ of $D$ of
dimension $l-1$ and elements $\lambda_1, \dots , \lambda_s$ in $k$
such that 
\[ 
         \pi (u) = \pi ( \sum_{i=1}^s \lambda_i x_{\tau_i}).
\] 
Taking into account that the field $k_1$
has characteristic $2$ and combining  Proposition~\ref{prop!differoddnpfofmainthm} with 
Remark~\ref{rem!aboutcharacterstic2fifferentialoperators}, we have
\begin{align*}
          (\partial_{\sigma} \circ \Psi \circ \pi )   (u^2) 
             &=  (\partial_{\sigma} \circ \Psi \circ \pi )   ( \sum_{i=1}^s \lambda_i^2 x_{\tau_i}^2 )  
                     =\sum_{i=1}^s \lambda_i^2  \big(    (\partial_{\sigma} \circ \Psi \circ \pi ) (x_{\tau_i}^2 ) \big)  \\
        & = \sum_{i=1}^s \lambda_i^2  \big( (  \Psi \circ \pi )   (x_{\sigma}x_{\tau_i}) \big)^2 
               =  \big(  \sum_{i=1}^s \lambda_i  ( (  \Psi \circ \pi )   (x_{\sigma}x_{\tau_i})) \big)^2      \\
   &   =  \big(  (  \Psi \circ \pi )   ( \sum_{i=1}^s   \lambda_ix_{\tau_i} x_{\sigma}) \big)^2     
            = \big( (  \Psi \circ \pi )   (x_{\sigma}u) \big)^2.
\end{align*}
\end{proof}

\begin{remark} \label {rem!notationalconveítion1}
 If we abuse the notation by avoiding writing down the maps $\Psi$ and $\pi$,
Equations~(\ref{eqn!relatingpartialxtausquarewith})   and  (\ref{eqn!relatingpartialusquarewith})
take the simpler form
\[
           \partial_{\sigma} ( x_{\tau}^2 )  =        (x_{\sigma}x_{\tau})^2 
                     \quad   \quad \quad \text{and}  \quad  \quad  \quad
           \partial_{\sigma} ( u^2 )  =        (x_{\sigma}u)^2
\] 
respectively.
\end{remark}

\begin {example}  \label {example!about1dimcaseno2}
We use the  assumptions of  Example~\ref{example!about1dimcaseno1}
and the notational convention described in Remark~\ref{rem!notationalconveítion1}.
We have
\[
       x_{2}^2   =      \frac {[1,3]} {[1,2]    [2,3] },  \quad   \quad
         \partial_{\{1\}} ( x_{2}^2 )  =      \frac {1} {[1,2]^2 }  = (x_{1}x_{2})^2 , \quad  \quad
          \partial_{\{2\}} ( x_{2}^2 )  =  \frac {[1,3]^2} {[1,2]^2    [2,3]^2 } =  ( x_{2}^2 )^2.
\]
Assume, in addition, that $m \geq 4$.  Then
\[
          \partial_{\{4\}} ( x_{2}^2 )   = 0  =   (x_{4}x_{2})^2. 
\]
\end {example}

\begin{corollary}  \label{cor!moioimofdf} 
Assume $u$ is a homogeneous element of $R$ of degree less or equal than 
$l$  such that $\pi (u) \not= 0$.
We then have that  $(\pi (u))^2 \not= 0$.
\end{corollary}

\begin{proof} 
 Using Remark~\ref{rem!basicpropertiesofA}, $A$ is Artinian, Gorenstein 
and  standard graded  with socle degree equal to  $n+1$.  It follows, 
by Remark~\ref{rem!poincare_duality_for_gor},
that there exists a  homogeneous   element  $h \in R$  of degree $\; l - \deg (u)\;$
such that  $\pi (uh)  \not= 0$. Combining  
 Proposition~\ref{prop!generationbysquarefree} with
Remark~\ref{rem!poincare_duality_for_gor},
there exists a face $\sigma$ of
$D$ of dimension $l-1$ such that $\pi (x_{\sigma}uh)  \not= 0$.

This implies that $(\Psi \circ \pi) (x_{\sigma}uh) \not= 0 $, hence, 
by Corollary~\ref{cor!differoddnpfofmainthm}, 
$(\Psi \circ \pi) ((uh)^2)  \not= 0$. Since $\pi$ is a $k$-algebra homomorphism,
we get  $(\pi (u))^2 \not= 0$.
\end{proof}

\subsection {Case $n$ is even} Assume $n \geq 2$ is even. We set $l= n/2$.
We assume $\sigma \in D$ is a face  of dimension $l-1$ and that 
$p$ is vertex of $D$  such that $\sigma \cup \{ p \}$ is a face of $D$ of 
dimension $l$.  We denote, in increasing order,  the elements of $\sigma$  
by $\sigma(1),\sigma(2), \dots , \sigma(l)$. 
 We define 
$\partial_{p, \sigma} : k_2 \to k_2 $ to be the $(n+1)$-th order differential operator
which is differentiation with respect to the variables in the set
\[
    \{ a_{1,p} \} \;  \cup  \; \{ a_{i, \sigma(j)} \;  \;: \;   \; 
                   2 \leq i \leq n+1,    \;  j = [ i/2 ]        \},
\]
where $[x]$  denotes the integral part of the real number $x$.

\begin{proposition}  \label{prop!differforevennpfofmainthm} 
Assume $\tau$ is a face of $D$ of dimension $l-1$ which does not contain $p$. We then have  
\begin{equation}   \label{eqn!bcdef1}
          (\partial_{p,\sigma} \circ \Psi \circ \pi )   (x_{\tau}^2x_p) =
              \big( (  \Psi \circ \pi )   (x_{\sigma}x_{\tau}x_p) \big)^2.
\end{equation}
\end{proposition}

\begin{proof}  
We set $\tau_1 = \tau \cup \{ p \}$.
If $ \tau_1$ is not a face of $D$, we have 
$ \pi   (x_{\tau}x_p) = 0$  and the proposition is true. 

Hence, we can assume that $ \tau_1 $ is  a face of $D$.
   We define the  sets
\[
     \mathcal{K}_1 = \{\eta \in F(D) : \tau_1 \subset \eta \}, \quad \quad 
     \mathcal{K}_2 = \{\eta \in F(D) : \tau_1 \cup  \sigma  \subset \eta  \}.
\]

We set $\gamma_1 = \tau_1 \cap \sigma$, 
$\gamma_2 = (\tau_1 \cup \sigma) \setminus \gamma_1$.
Since $p$ is not an element of $\sigma$, we have $\gamma_1 = \tau \cap \sigma$.
Using Theorem~\ref{thm!aboutlsquares}, we get
\[
     (\Psi  \circ   \pi)  (x_\tau^2x_p)   =
              \sum_{ \eta \in  \mathcal{K}_1 }   H( \tau, \{ p \} , \eta) \quad
   \quad   \text{and}  \quad   \quad 
     (\Psi  \circ   \pi)  (x_{\tau} x_{\sigma}x_p)   =
              \sum_{ \eta \in  \mathcal{K}_2 }   H( \gamma_1,  \gamma_2 , \eta).
\]
Clearly $\mathcal{K}_2 \subset \mathcal{K}_1$.
If $\eta \in   \mathcal{K}_1 \setminus \mathcal{K}_2$, 
we have that $ \sigma \setminus (\eta \cap \sigma ) \not= \newemptyset$, which implies that  $ \partial_{p,\sigma} ( H( \tau, \{ p \} , \eta)) = 0$.   Hence
\[
     (\partial_{p,\sigma} \circ \Psi  \circ   \pi)  (x_\tau^2x_p)   =
              \sum_{ \eta \in  \mathcal{K}_2 }   \partial_{p,\sigma} (H( \tau, \{ p \}, \eta)).
\]
Since the field $k_1$ has characteristic $2$
\[
    ( (\Psi  \circ   \pi)  (x_{\tau} x_{\sigma}x_{p}))^2   =
              \sum_{ \eta \in  \mathcal{K}_2 }   (H( \gamma_1,  \gamma_2 , \eta))^2.
\]

Assume that $\eta \in \mathcal{K}_2$.   Using   Corollary~\ref{cor!evenaeoeoerd}, we have  
\begin{align*}
     \partial_{p,\sigma} (H( \tau,  \{ p \}, \eta))   & =     \partial_{p,\sigma} \big(
                                \frac {  \prod_{i \in \tau}   M(  (\eta \cup \{ r \} )  \setminus \{  i \})   }  { 
                                      M(\eta) \cdot \prod_{i \in \eta \setminus (\tau \cup \{ p \})}  M( ( \eta \cup \{ r \} ) \setminus \{ i \}) } \big)  \\
      & =   \frac {   \prod_{i \in \tau \cap \sigma }  (M(  (\eta \cup \{ r \} )  \setminus \{  i \}))^2  }  { 
                                      (M(\eta))^2  \cdot  \prod_{i \in \eta \setminus (\tau \cup \{ p \} \cup \sigma)} (M( ( \eta \cup \{ r \} ) \setminus \{ i \}))^2 } \\
     &=    (H( \gamma_1,  \gamma_2 , \eta))^2,
\end{align*}
which finishes the proof. 
\end{proof}

\begin{remark}  Conjecture~\ref{conj!generalisingthmcomputingdifferoperators} 
contains  a conjectural statement generalising Proposition~\ref{prop!differforevennpfofmainthm}. 
\end{remark}

We will need the following strengthening of  Proposition~\ref{prop!generationbysquarefree}.

\begin{proposition}  \label{prop!dstrengtheningofsquarefree} 
Assume    $1 \leq d \leq n+1$  and $u \in R_d$. Then there exist $s>0$,
faces $\tau_1, \dots , \tau_s$ of $D$  dimension $d-1$ and elements $\lambda_1, \dots , \lambda_s$ in $k$
such that 
\[ 
         \pi (u) = \pi ( \sum_{i=1}^s \lambda_i x_{\tau_i})
\] 
and, moreover, $p$ is not an element of $\tau_i$ for all $1 \leq i \leq s$.
\end{proposition}

\begin {proof}
Using Proposition~\ref{prop!generationbysquarefree}, it is enough 
to assume that $u = x_{\eta}$, where $\eta$ is a face of $D$ of dimension $d-1$.
If $p$ is not an element of $\eta$, the result is obvious by setting
$s=1$, $\tau_1= \eta$, $\lambda_1= 1$.

Assume now $p \in \eta$.
Without loss of generality, we can assume $p=1$ and $\eta = \{1,2, \dots , d \}$.
By Proposition~\ref{prop!aboutGausselim},  we have
\[
        \sum_{t =1}^m [ 2,3, \dots , n+1,t] \pi (x_t) = 0.
\]
Hence,
\[
       \pi (x_1) = 
        - \sum_{t =n+2}^m \frac { [ 2,3, \dots , n+1,t]}{[ 2,3, \dots , n+1,1] } \pi (x_t),
\]
which implies that 
\[
       \pi (x_{\eta}) = 
        - \sum_{t =n+2}^m \frac { [ 2,3, \dots , n+1,t]}{[ 2,3, \dots , n+1,1] } \pi (x_t  \prod_{i=2}^d x_i).
\]
The result follows.
\end{proof}

\begin{corollary}  \label{cor!differevennpfofmainthm} 
Assume $u$ is a homogeneous element of $R$ of degree $l$. 
We then have  
\begin{equation}  \label{eqn!bcdef2}
          (\partial_{p,\sigma} \circ \Psi \circ \pi )   (u^2x_p) =
              \big( (  \Psi \circ \pi )   (x_{\sigma}ux_p) \big)^2.
\end{equation}
\end{corollary}

\begin{proof}
 Using  Proposition~\ref{prop!generationbysquarefree},
there exist  $s > 0$,  faces $\tau_1, \dots , \tau_s$ of $D$ of
dimension $l-1$ and elements $\lambda_1, \dots , \lambda_s$ in $k$
such that 
\[ 
         \pi (u) = \pi ( \sum_{i=1}^s \lambda_i x_{\tau_i})
\] 
and, moreover,  $p$ is not an element of $\tau_i$ for all $1 \leq i \leq s$.

Taking into account that the field $k_1$
has characteristic $2$ and combining  Proposition~\ref{prop!differforevennpfofmainthm}
with 
Remark~\ref{rem!aboutcharacterstic2fifferentialoperators}, we have
\begin{align*}
          (\partial_{p,\sigma} \circ \Psi \circ \pi )   (u^2x_p) 
             &=  (\partial_{p,\sigma} \circ \Psi \circ \pi )   ( \sum_{i=1}^s \lambda_i^2 x_{\tau_i}^2x_{p} )  
                     =\sum_{i=1}^s \lambda_i^2  \big(    (\partial_{p,\sigma} \circ \Psi \circ \pi ) (x_{\tau_i}^2x_{p} ) \big)  \\
        & = \sum_{i=1}^s \lambda_i^2  \big( (  \Psi \circ \pi )   (x_{\sigma}x_{\tau_i}x_{p}) \big)^2                
               =  \big(  \sum_{i=1}^s \lambda_i  ( (  \Psi \circ \pi )   (x_{\sigma}x_{\tau_i}x_{p})) \big)^2      \\
   &   =  \big(  (  \Psi \circ \pi )   ( \sum_{i=1}^s   \lambda_ix_{\tau_i} x_{\sigma}x_{p}) \big)^2     
   =    \big( (  \Psi \circ \pi )   (x_{\sigma}ux_{p}) \big)^2.
\end{align*}
\end{proof}

\begin{remark} \label {rem!notationalconveítion1bcdef}
 If we abuse the notation by avoiding writing down the maps $\Psi$ and $\pi$,
Equations~(\ref{eqn!bcdef1})   and  (\ref{eqn!bcdef2})
take the simpler form
\[
           \partial_{p,\sigma} ( x_{\tau}^2x_{p})  =        (x_{\sigma}x_{\tau}x_{p})^2 
                     \quad   \quad  \quad  \text{and}  \quad  \quad  \quad 
           \partial_{p,\sigma} ( u^2 x_{p})  =        (x_{\sigma}ux_{p})^2
\] 
respectively.
\end{remark}

\begin {example}  \label {example!about2dimcase}
Assume $D$ is the boundary complex of the $3$-simplex with vertex set
$\{1,2,3,4\}$. We set $p=1, \tau = \{ 2 \}$. Using 
Corollary~\ref{cor!oneSquareequationcorollary} 
and the notational convention described in Remark~\ref{rem!notationalconveítion1bcdef},
we have
\[
        x_{\tau}^2x_{p}   =      \frac {[1,3,4]} {[1,2,3]    [1,2,4] },  \quad   \quad
         \partial_{p, \{2\}} (  x_{\tau}^2x_{p} )  =      \frac {[1,3,4]^2} {[1,2,3]^2    [1,2,4]^2 }  
                                     = (x_2  x_{\tau}x_{p})^2  
\]
and
\[
         \partial_{p, \{3\}} (  x_{\tau}^2x_{p} )  =       \frac {1} { [1,2,3]^2 }  
                                     = (x_3x_{\tau}x_{p})^2,   \quad   \quad
        \partial_{p, \{4\}} (  x_{\tau}^2x_{p} )  =       \frac {1} { [1,2,4]^2 }  
                                     = (x_4x_{\tau}x_{p})^2.
\]
\end {example}

\begin{corollary}  \label{cor!pdgfsaffsa} 
Assume $u$ is a homogeneous element of $R$ of degree less or equal than 
$l$  such that $\pi (u) \not= 0$.
We then have that  $(\pi (u))^2 \not= 0$.
\end{corollary}

\begin{proof} 

Using Remark~\ref{rem!basicpropertiesofA}, $A$ is Artinian, Gorenstein 
and  standard graded  with socle degree equal to  $n+1$.  It follows, 
by Remark~\ref{rem!poincare_duality_for_gor},
that there exists a  homogeneous   element  $h \in R$  of degree $\; l - \deg (u)\;$
such that  $\pi (uh)  \not= 0$. Combining  
 Proposition~\ref{prop!generationbysquarefree} with
Remark~\ref{rem!poincare_duality_for_gor},
there exists a face $\sigma_1$ of
$D$ of dimension $l$ such that $\pi (x_{\sigma_1}uh)  \not= 0$.

We fix an element $p$ of  $\sigma_1$, and set $\sigma = \sigma_1 \setminus \{ p \}$.
Therefore, $\pi (x_{\sigma_1}uh)  \not= 0$ implies that 
$ (\Psi \circ \pi) (x_{\sigma}uhx_{p}) \not= 0 $.  Using
Corollary~\ref{cor!differevennpfofmainthm}, it follows that
$ (\Psi \circ \pi) ((uh)^2x_p) )\not= 0$. Since $\pi$ is a $k$-algebra homomorphism,
we get  $(\pi (u))^2 \not= 0$.
\end{proof}

\subsection{Proof of Theorem~\ref{thm!anisotropicityinchar2}} \label{subs!proofofmainthm}

We now prove Theorem~\ref{thm!anisotropicityinchar2}. If $n$ is odd, it follows from
Corollary~\ref{cor!moioimofdf}, while if $n$ is even, it follows from 
Corollary~\ref {cor!pdgfsaffsa}.

\section{The differential operator for $n$ odd} \label{sec!diffoperatornodd}

The aim of the present section is to establish, in conjuction
with the following 
two Sections~\ref{sec!diffoperatorneven} and~\ref{sec!diffoperatoridentities}, 
the results about the differential operators
that were used in Section~\ref{sec!usingthefdifferentialoperatorsto}.

In the present section we work over a field $k_1$  of characteristic $2$.

Assume $n \geq 1$ is odd  and $m$ is an integer with $m \geq n+1 $.
 We set  $Z = m + 2n$  and 
denote by  $M$ the $(n+1) \times Z$ matrix whose
$(i,j)$-entry is equal to the variable $a_{i,j}$, for
$1 \leq i  \leq n$ and  $1 \leq j  \leq Z$.  Given 
a subset $\mathcal{A}$ of the set $\{1,2, \dots , Z\}$
of cardinality $n+1$, we denote by  $M(\mathcal{A})$
the determinant of the $(n+1) \times (n+1) $ submatrix
of $M$ obtained by keeping the columns of $M$
specified by the set   $\mathcal{A}$.

We denote by $k_2$ the field of fractions of the polynomial
ring
\[
             k_1 [ a_{i,j} :  1 \leq i \leq n+1,  \;  1 \leq j \leq Z ].
\]
We set $l = (n+1)/2$.  Assume
\[
    \tau_1 = \{ c_1,  \dots , c_{l} \} , \quad   \tau_2 = \{ g_1,  \dots , g_{l+1} \}
\]
are two subsets of the  set  $\{ 1,2, \dots , Z \} $ 
such that $\tau_1 \cup \tau_2$ has  cardinality $2l+1$. 

We set  $\tau = \tau_1 \cup \tau_2$ and  
\[
       G(    \tau_1 ,  \tau_2)  = \frac {
                        \prod_{i=1}^{l}  M( \tau \setminus \{ c_i \})  }  { 
                        \prod_{i=1}^{l+1}  M( \tau \setminus \{ g_i \}) }.
\] 

For the rest of this section 
we make the assumption that $\tau$ is a subset of the set 
$\{ 1,2, \dots  m \}$. 
We fix $r$ with $m+1 \leq r \leq Z$
and set,  for $1 \leq i \leq l+1$,
\[
       G_i ( \tau_1 ,  \tau_2, \{ r \} )  =
              G ( \tau_1 ,  ( \tau_2 \cup \{ r \}) \setminus \{ g_i \} ).  
\] 
We denote by $G^{sp}_i  ( \tau_1 ,  \tau_2)$ the result of substituting
in $G_i ( \tau_1 ,  \tau_2,   \{ r \} )$ the value  $1$ for the variable $a_{1,r}$ and
 the value  $0$ for the variables $a_{j,r}$, for $2 \leq j \leq n+1$. We remark that 
 $G^{sp}_i  ( \tau_1 ,  \tau_2)$ is well-defined, since 
the denominator of    $G_i ( \tau_1 ,  \tau_2,  \{ r \} )$ does not
vanish when we perform the substitution.

Moreover, we denote by $\; T_{n+1}  :  k_2 \to k_2 \;$ 	 the $(n+1)$-th order differential
operator which is differentiation  with respect to
the set of variables
\[
      \{ a_{1,c_1}, a_{2,c_1},\; a_{3,c_2}, a_{4,c_2},
               \; a_{5,c_3}, a_{6,c_3}, \dots ,
               a_{n,c_l}, a_{n+1,c_l}      \}.
\]

\begin{remark}  This set of variables can also be described as the set
\[
      \{ a_{i,c_j} : 1 \leq i \leq n +1 ,   j =  [(i+1)/2]  \},
\]
where $[x]$  denotes the integral part of the real number $x$.  For example, 
if $n=3$, then
\[
         T_{n+1} = \frac { \partial^4}{\partial a_{1,c_1} \; \partial a_{2,c_1} \; \partial a_{3,c_2} \; \partial  a_{4,c_2}}.
\]
\end{remark} 

\vspace {10pt}

We remind the reader  that the field $k_1$  has characteristic $2$.

\begin{proposition}
\label{prop!sumofGafornodd} 
We have the following equality in the field $k_2$
\[
       G(    \tau_1 ,  \tau_2)  = \sum_{i=1}^{l+1}   G_i ( \tau_1 ,  \tau_2, \{ r \}).
\] 
\end{proposition}

\begin{proof}    Denote by $D$ the boundary complex of the simplex of dimension $n+1$
with vertex set  $\tau$.
By  Proposition~\ref{prop!simplexnoddformula},  we have
\[
       (\Psi  \circ   \pi)    (\prod_{i=1}^{l} x_{c_i}^2)  =   G(    \tau_1 ,  \tau_2).
\]
Since $G_i ( \tau_1 ,  \tau_2, \{ r \}) = H( \tau_1 ,  \newemptyset,   \tau \setminus \{ g_i\}) $,
by Theorem~\ref{thm!aboutlsquares} we have
\[
       (\Psi  \circ   \pi)    (\prod_{i=1}^{l} x_{c_i}^2)  =   \sum_{i=1}^{l+1}   G_i ( \tau_1 ,  \tau_2, \{ r \}).
\]
The result follows.
\end{proof}

For an example related to the above Proposition~\ref{prop!sumofGafornodd} see Example~\ref{example!about1dimcaseno1}.

The following corollary follows immediately from 
Proposition~\ref{prop!sumofGafornodd}, by taking into
account that, for all  $1 \leq j \leq n+1$, the  variable $a_{j,r}$
does not appear in  $G(  \tau_1 ,  \tau_2)$.

\begin{corollary}
\label{cor!sumofGafornoddwere} 
We have the following equality in the field $k_2$
\[
       G(    \tau_1 ,  \tau_2 )  = 
                \sum_{i=1}^{l+1}   G^{sp}_i ( \tau_1 ,  \tau_2  ).
\] 
\end{corollary}

The following proposition is an immediate corollary of 
Part 1 of Theorem~\ref{theorem!computingdifferoperators}. 
For simplicity of notation, for $i \in  \tau$ we set 
$M_i =  M( \tau \setminus \{ i \})$.

\begin{proposition}
\label{prop!simplexequalityfornodd} 
We have the following equality in the field $k_2$
\[
        T_{n+1} \; (  \prod_{i \in \tau }   M_i  )= 
                 \prod_{i \in \tau_1}   (M_i)^2 . 
\] 
\end{proposition}

\begin{corollary}
\label{cor!oddfgshffosdfs} 
Assume $S$ is a subset of $\tau$.  
We then have the following equality in the field $k_2$
\[
        T_{n+1} \big( \frac  {\prod_{i \in S}  M_i}
              {\prod_{i \in \tau \setminus S }  M_i}  \big) =
          \frac  {\prod_{i \in S \cap \; \tau_1}  (M_i)^2} 
             {\prod_{i \in \tau_2 \setminus S} (M_i)^2}.
\] 
\end{corollary}

\begin{proof} 
  Using  Proposition~\ref{prop!simplexequalityfornodd}  and Remark~\ref{rem!aboutcharacterstic2fifferentialoperators},
we have
\begin{align*}
        T_{n+1} \big( \frac  {\prod_{i \in S}  M_i} {\prod_{i \in \tau \setminus S }  M_i} \big)  &=
             T_{n+1} \big( \frac  {\prod_{i \in \tau}  M_i} {(\prod_{i \in \tau \setminus S }  M_i)^2} \big) 
             =         \frac { T_{n+1} ( \prod_{i \in \tau}  M_i)}  {(\prod_{i \in \tau \setminus S }  M_i)^2}  \\
       &  =
             \frac { \prod_{i \in \tau_1 }   (M_i)^2} {\prod_{i \in \tau \setminus S}
 (M_i)^2}
   =  \frac  { E \cdot  \prod_{i \in S \cap \; \tau_1}  (M_i)^2} {E \cdot \prod_{i \in \tau_2 \setminus S} (M_i)^2},
\end{align*}
where $\;               E = \prod_{i \in \tau_1 \setminus S }  (M_i)^2$.   The result follows.
\end{proof}

\section{The differential operator for $n$ even} \label{sec!diffoperatorneven}

The aim of the present section is to establish, in conjuction
with the previous
Section~\ref{sec!diffoperatornodd} and 
the following Section~\ref{sec!diffoperatoridentities}, 
the results about the differential operators
that were used in Section~\ref{sec!usingthefdifferentialoperatorsto}.

In the present section we work over a field $k_1$  of characteristic $2$.

Assume $n \geq 1$ is even  and $m$ is an integer with $m \geq n+1 $.
 We set  $Z = m + 2n$  and 
denote by  $M$ the $(n+1) \times Z$ matrix whose
$(i,j)$-entry is equal to the variable $a_{i,j}$, for
$1 \leq i  \leq n$ and  $1 \leq j  \leq Z$.  Given 
a subset $\mathcal{A}$ of the set $\{1,2, \dots , Z\}$
of cardinality $n+1$, we denote by  $M(\mathcal{A})$
the determinant of the $(n+1) \times (n+1) $ submatrix
of $M$ obtained by keeping the columns of $M$
specified by the set   $\mathcal{A}$.

We denote by $k_2$ the field of fractions of the polynomial
ring
\[
             k_1 [ a_{i,j} :  1 \leq i \leq n+1,  \;  1 \leq j \leq Z ].
\]
We set $l = n/2$.  Assume
\[
    \tau_1 = \{ c_1,  \dots , c_{l} \} , \quad
     \tau_2 = \{ b  \},  \quad 
    \tau_3 = \{ g_1,  \dots , g_{l+1} \}
\]
are three subsets of the  set  $\{ 1,2, \dots , Z \} $ 
such that $ \cup_{i=1}^3  \tau_i$ has  cardinality $2l+2$. 

We set  $\tau = \cup_{i=1}^3  \tau_i$ and  
\[
       G(    \tau_1 ,  \tau_2, \tau_3 )  = \frac {
                        \prod_{i=1}^{l}  M( \tau \setminus \{ c_i \})  }  { 
                        \prod_{i=1}^{l+1}  M( \tau \setminus \{ g_i \}) }.
\] 

For the rest of this section 
we make the assumption that $\tau$ is a subset of the set 
$\{ 1,2, \dots  m \}$. 
We fix $r$ with $m+1 \leq r \leq Z$
and set,  for $1 \leq i \leq l+1$,
\[
       G_i ( \tau_1 ,  \tau_2,  \tau_3,  \{ r \} )  =
              G ( \tau_1 , \tau_2,   ( \tau_3 \cup \{ r \}) \setminus \{ g_i \} ).  
\] 
We denote by $G^{sp}_i  ( \tau_1 ,  \tau_2,  \tau_3)$ the result of substituting
in $G_i ( \tau_1 ,  \tau_2,  \tau_3,  \{ r \} )$ the value  $1$ for the variable $a_{1,r}$ and
 the value  $0$ for the variables $a_{j,r}$, for $2 \leq j \leq n+1$. We remark that 
 $G^{sp}_i  ( \tau_1 ,  \tau_2,  \tau_3)$ is well-defined, since 
the denominator of    $G_i ( \tau_1 ,  \tau_2,  \tau_3,  \{ r \} )$ does not
vanish when we perform the substitution.

Moreover, we denote by $\; T_{n+1} : k_2 \to k_2 \;$ 	 the $(n+1)$-th order differential
operator which is differentiation  with respect to
the set of variables
\[
      \{ a_{1,b},  \; a_{2,c_1}, a_{3,c_1},
               \; a_{4,c_2}, a_{5,c_2}, \dots ,
               a_{n,c_l}, a_{n+1,c_l}      \}.
\]

\begin{remark}  This set of variables can also be described as the set
\[
      \{ a_{1,b}   \}   \cup    \{ a_{i,c_j} : 2 \leq i \leq n +1 ,   j =  [i/2]   \},
\]
where $[x]$  denotes the integral part of the real number $x$.
 For example,  if $n=2$, then
\[
         T_{n+1} = \frac { \partial^3}{\partial a_{1,b} \; \partial a_{2,c_1} \; \partial a_{3,c_1}}.
\]
\end{remark} 

\vspace {10pt}

We remind the reader  that the field $k_1$  has characteristic $2$.

\begin{proposition}  \label{prop!sumofGaforneven} 
We have the following equality in the field $k_2$
\[
       G(    \tau_1 ,  \tau_2, \tau_3)  = 
                \sum_{i=1}^{l+1}   G_i ( \tau_1 ,  \tau_2, \tau_3,  \{ r \}).
\] 
\end{proposition}

\begin{proof}    Denote by $D$ the boundary complex of the simplex of dimension $n+1$
with vertex set  $\tau$.
By  Proposition~\ref{prop!simplexnevenformula},  we have
\[
       (\Psi  \circ   \pi)    (x_{b} \prod_{i=1}^{l} x_{c_i}^2)  =   G(    \tau_1 ,  \tau_2, \tau_3 ).
\]
Since $G_i ( \tau_1 ,  \tau_2,  \tau_3, \{ r \}) = H( \tau_1 , \tau_2 ,   \tau \setminus \{ g_i\}) $,
by Theorem~\ref{thm!aboutlsquares} we have
\[
       (\Psi  \circ   \pi)    (x_{b} \prod_{i=1}^{l} x_{c_i}^2)  =   \sum_{i=1}^{l+1}   G_i ( \tau_1 ,  \tau_2,   \tau_3 ,  \{ r \}).
\]
The result follows.
\end{proof}

The following corollary follows immediately from 
Proposition~\ref{prop!sumofGaforneven}, by taking into
account that, for all  $1 \leq j \leq n+1$, the  variable $a_{j,r}$
does not appear in  $ G(    \tau_1 ,  \tau_2, \tau_3) $.

\begin{corollary}
\label{cor!sumofGafornevenwere} 
We have the following equality in the field $k_2$
\[
       G(    \tau_1 ,  \tau_2, \tau_3)  = 
                \sum_{i=1}^{l+1}   G^{sp}_i ( \tau_1 ,  \tau_2, \tau_3 ).
\] 
\end{corollary}

The following proposition is an immediate corollary of 
Part 2  of  Theorem~\ref{theorem!computingdifferoperators}.
For simplicity of notation, for $i \in  \tau_1 \cup \tau_3$ we set 
$M_i =  M( \tau \setminus \{ i \})$.

\begin{proposition}
\label{prop!simplexequalityforneven} 
We have the following equality in the field $k_2$
\[
        T_{n+1} \; \big(  \prod_{i \in \tau_1 \cup \tau_3 }   M_i  \; \big)   =  \prod_{i \in \tau_1 }   (M_i)^2. 
\] 
\end{proposition}

\begin{corollary}
\label{cor!evenaeoeoerd} 
Assume $S$ is a subset of $\tau_1 \cup  \tau_3$.  
We then have the following equality in the field $k_2$
\[
        T_{n+1} \big( \frac  {\prod_{i \in S} 
                           M_i} {\prod_{i \in (\tau_1 \cup \tau_3) \setminus S }  M_i} \big) =
          \frac  {\prod_{i \in S \; \cap \; \tau_1} 
                         (M_i)^2} {\prod_{i \in \; \tau_3 \setminus S} (M_i)^2}.
\] 
\end{corollary}

\begin{proof}  We set  $w = \tau_1 \cup \tau_3$.
  Using  Proposition~\ref{prop!simplexequalityforneven}  and Remark~\ref{rem!aboutcharacterstic2fifferentialoperators},
we have
\begin{align*}
        T_{n+1} \big( \frac  {\prod_{i \in S}  M_i} {\prod_{i \in w \setminus S }  M_i} \big)  &=
             T_{n+1} \big( \frac  {\prod_{i \in w}  M_i} {(\prod_{i \in w \setminus S }  M_i)^2} \big) 
             =         \frac { T_{n+1} ( \prod_{i \in w}  M_i)}  {(\prod_{i \in w \setminus S }  M_i)^2}  \\
       &  =
             \frac { \prod_{i \in \tau_1 }   (M_i)^2} {\prod_{i \in w \setminus S} (M_i)^2}
   =  \frac  { E \cdot  \prod_{i \in S \cap \; \tau_1}  (M_i)^2} {E \cdot \prod_{i \in \tau_3  \setminus S} (M_i)^2},
\end{align*}
where $\;               E = \prod_{i \in \tau_1 \setminus S }  (M_i)^2$.   The result follows.
\end{proof}

\section{Some useful characteristic $2$ identities } \label{sec!diffoperatoridentities}

The aim of the present section is to establish, in conjuction with
the previous two
Sections~\ref{sec!diffoperatornodd} and ~\ref{sec!diffoperatorneven}, 
the results about the differential operators
that were used in Section~\ref{sec!usingthefdifferentialoperatorsto}.

In the present section we work over a field $k_1$  of characteristic $2$.

Assume  $h \geq 2$ is
an integer.    We denote by $k$ the field of fractions of the polynomial
ring
\[
             k_1 [ a_{i,j} :  1 \leq i \leq h+2,  \;  1 \leq j \leq h+1 ].
\]
We denote by  $M^{big}$ the $(h+2) \times (h+1)$ matrix whose
$(i,j)$-entry is equal to the variable $a_{i,j}$, for
$1 \leq i  \leq h+2$ and  $1 \leq j  \leq h+1$.

Assume $h \geq 2$ is even. We denote by $N^{(h)}$ the $h \times (h+1)$ submatrix 
of $M^{big}$, obtained by keeping the rows indexed by
$1,2, \dots , h$. 
We denote by $P^{(h)}$ the $h \times (h+1)$   submatrix 
of $M^{big}$, obtained by keeping the rows indexed by
$3,4, \dots , h+2$.  
We define the folowing two sets of variables
\[
  \mathcal{A}_{N,h}  =  \{ a_{i,j} : 1 \leq i \leq h,   j =  [(i+1)/2]  \}, \quad
  \mathcal{A}_{P,h}  =      \{ a_{i,j} : 3 \leq i \leq h+2 ,   j =  [(i+1)/2]  \},
\]
where $[x]$  denotes the integral part of the real number $x$.
For $S \in \{ N, P \}$, we denote by $T_{S,h}$ the $h$-th order differential operator
which is partial differentiation with respect to the variables in
the set $\mathcal{A}_{S,h}$.

Assume $h \geq 3$ is odd. We denote by $Q^{(h)}$ the $h \times (h+1)$ submatrix 
of $M^{big}$, obtained by keeping the rows indexed by
$2,3, \dots , h+1$. 
We define the folowing set of variables
\[
  \mathcal{A}_{Q,h}  = \{ a_{2,1}  \} \cup     \{ a_{i,j} : 3 \leq i \leq h+1 ,   j =  [(i+1)/2]  \}.
\]
We denote by $T_{Q,h}$ the $h$-th order differential operator
which is partial differentiation with respect to the variables in
the set $\mathcal{A}_{Q,h}$.

In the present section we will use the following notational convention.    
Assume  $l \geq 1$,  $S$ is 
an $ l    \times (l+1)$ matrix
and $1 \leq i \leq l+1$. We will denote by $S_i$  
the determinant of the $l \times l$ submatrix
of $S$ obtained by deleting  the $i$-th column of $S$.

\begin{remark}    \label{rem!aboutcharacterstic2fifferentialoperators}
We will use that, since the field $k_1$ has characteristic $2$, we have
\[
      T_{S,h}  ( f^2 g)  =      f^2       T_{S,h}  ( g)
\]
for all   $f,g  \in k$,  $S \in  \{ N, P, Q  \}$  and $h \geq 2$ as above
(that is, $h$ even if $S=N$ or $S=P$ and  $h$ odd if $S=Q$).
Indeed, by the Leibnitz Rule,
\[
    \frac{\partial   } {\partial a_{i,j} } (f^2 g)   =       g   \frac{\partial   } {\partial a_{i,j} } (f^2)  +
                  f^2  \frac{\partial   } {\partial a_{i,j} } (g)   =
      2 g  f  \frac{\partial   } {\partial a_{i,j} } (f)  +
                  f^2  \frac{\partial   } {\partial a_{i,j} } (g) =                   f^2  \frac{\partial   } {\partial a_{i,j} } (g),
\]
and $  T_{S,h}$ is a composition of such operators. Consequently,   if $f,g  \in k$ with $g \not= 0$, then
\[
      T_{S,h} ( \frac{ f}{g})  =   T_{S,h}  (\frac{ fg}{g^2})   =  \frac  {T_{S,h}  ( fg) } {g^2}.
\]
\end{remark}

\begin{proposition} \label{prop!Prop1fordifferoperators} 
  Assume that  $h \geq 2$ is even and that 
\[
       T_{N,h} (  \prod_{i=1}^{h+1}  N^{(h)}_i )    =     \prod_{i=1}^{h/2} ( N^{(h)}_i)^2.
\] 
We then have 
\[
       T_{P,h} (  \prod_{i=1}^{h+1}  P^{(h)}_i )    =     \prod_{i=2}^{(h+2)/2} ( P^{(h)}_i)^2.
\] 
\end{proposition}

\begin {proof}  We  denote by  $N^{mod}$ the matrix obtained from $N^{(h)}$ by
putting the last column of $N^{(h)}$ first. 
Taking into account that the field $k_1$ has characteristic $2$, we get
$N^{mod}_1= N^{(h)}_{h+1}$  and that
\[  
        N^{mod}_i = N^{(h)}_{i-1},
\]
for all $2 \leq i \leq h+1$.  Hence,  using the assumption we have
\begin{equation}  \label{eqn!differentfornmod}
       T_{N,h} (  \prod_{i=1}^{h+1}  N^{mod}_i )    = T_{N,h} (  \prod_{i=1}^{h+1}  N^{(h)}_i ) =
             \prod_{i=1}^{ h/2} ( N^{(h)}_{i})^2   =  \prod_{i=1}^{ h/2} ( N^{mod}_{i+1})^2
            =  \prod_{i=2}^{(h+2)/2} ( N^{mod}_i)^2.  
\end{equation}

 We have that both  $N^{mod}$ and $P^{(h)}$
are $h \times (h+1)$ matrices.  The entries of each matrix are 
independent indeterminates.  For $1 \leq i \leq h$  and   $1 \leq j \leq h+1$,
we denote by $n_{i,j}$ the $(i,j)$-entry of  $N^{mod}$ and by
$p_{i,j}$ the $(i,j)$-entry of  $P^{(h)}$.   By definition,  $T_{N,h}$ is 
differentiation with respect to the variables in the set
\[
         \{ n_{i,j} : 1 \leq i \leq h ,   j =  1+ [(i+1)/2]  \},
\]
while  $T_{P,h}$ is 
differentiation with respect to the variables in the set
\[
         \{ p_{i,j} : 1 \leq i \leq h ,   j =  1+ [(i+1)/2]  \}.
\]
There exists a unique isomorphism of $k_1$-algebras 
$\; \phi : k_1 [n_{i,j}] \to  k_1 [p_{i,j}]$ such that 
$\phi (n_{i,j}) = p_{i,j}$  for all $1 \leq i \leq h$  and   $1 \leq j \leq h+1$.
As a consequence, the result follows from Equation~(\ref{eqn!differentfornmod}).
\end {proof}

\begin{proposition} \label{prop!Prop2fordifferoperators} 
  Assume $h=2$. We have 
\[
       T_{N,2} (   N^{(2)}_1 N^{(2)}_2 N^{(2)}_3  )    =      (  N^{(2)}_1)^2   \quad  \quad
          \text {and}  \quad \quad 
       T_{P,2} (   P^{(2)}_1 P^{(2)}_2 P^{(2)}_3  )    =      (  P^{(2)}_2)^2.
\] 
\end{proposition}

\begin {proof} 
Using Proposition~\ref{prop!Prop1fordifferoperators},  it is enough 
     to  prove only the first equality. 
        We have 
\[
    T_{N,2} = \frac {\partial^2}{\partial a_{1,1} \; \partial a_{2,1} }
\]
and 
\[
     N^{(2)}_1 =   \det \left(   \begin{array}{ccc}
                          a_{1,2}    &    a_{1,3}   \\
                          a_{2,2}    &    a_{2,3}      
        \end{array}         \right) , \quad    
     N^{(2)}_2 =    \det   \left(   \begin{array}{ccc}
                        a_{1,1}      &    a_{1,3}   \\
                        a_{2,1}      &    a_{2,3}      
        \end{array}         \right),    \quad 
     N^{(2)}_3 =    \det   \left(   \begin{array}{ccc}
                        a_{1,1}    &    a_{1,2}      \\
                        a_{2,1}    &    a_{2,2}         .
        \end{array}         \right).  
\]
The result follows by an easy direct computation, taking into account
that the field $k_1$ has characteristic $2$.
\end {proof}

\begin{remark} \label{rem!weneedchar2} 
It is easy to see that  the assumption that 
the field $k_1$ has characteristic $2$ is crucial in order to have
the equalities in the statement of Proposition~\ref{prop!Prop2fordifferoperators}.  
\end{remark} 

\begin{proposition} \label{prop!Prop3fordifferoperators} 
1)   Assume $h \geq 4$ is even and that
\[
       T_{P,h-2} (  \prod_{i=1}^{h-1} P^{(h-2)}_i )    =     \prod_{i=2}^{ h/2} (P^{(h-2)}_i)^2.
\] 
We then have 
\[
       T_{Q,h-1} (  \prod_{i=2}^{h} Q^{(h-1)}_i )    =     \prod_{i=2}^{h/2} (Q^{(h-1)}_i)^2.
\] 

2)   Assume $h \geq 4$ is even and that
\[
       T_{Q,h-1} (  \prod_{i=2}^{h} Q^{(h-1)}_i )    =     \prod_{i=2}^{h/2} (Q^{(h-1)}_i)^2.
\] 
We then have 
\[
       T_{N,h}  (  \prod_{i=1}^{h+1} N^{(h)}_i )    =     \prod_{i=1}^{h/2} (N^{(h)}_i)^2.
\] 
\end{proposition}

\begin {proof}   We fix an even integer $h \geq 4$. For simplicity of notation, we set 
\[
    N =  N^{(h)}, \; \; \;      T_{N} =  T_{N,h}, \; \;   \; 
    Q=  Q^{(h-1)}, \; \; \;       T_{Q} =  T_{Q,h-1}, \; \;    \; 
    P =   P^{(h-2)}, \;  \; \;      T_{P} =  T_{P,h-2}.   
\] 

\vspace{10pt}

  We first prove  Part 1).   We set 
\[
    W = \frac { \prod_{i=(h+2)/2}^{h-1} Q_i } { Q_h   \prod_{i=2}^{h/2} Q_i }. 
\]
Using   Remark~\ref{rem!aboutcharacterstic2fifferentialoperators}, it is enough to
prove that
\begin{equation}  \label {eqn!eqnforWandQ}
              T_Q  (  W  )    =   1/ (Q_h)^2.
\end{equation}

Denote by $r_1$ the transpose of the $1 \times (h-1)$ matrix  
$\; (1,0,0, \dots , 0)$.
For $\; s \in \{ 2, 3,  \dots ,h/2 \} \cup \{h\} \;$  we denote by $U^{<s>}$ the $ (h-1) \times h$ matrix
obtained by replacing 
the $s$-th column of $Q$   with  $r_1$.

We set, for $s \in \{2,3, \dots , h/2 \} \cup \{ h \}$,
\[
     W^{<s>}   =    \frac { \prod_{j=(h+2)/2}^{h-1} U^{<s>}_j } {U^{<s>}_h   \prod_{j=2}^{h/2} U^{<s>}_j }. 
\]
Since $ U^{<s>}_j \not= 0$, for all $1 \leq j \leq h$,  we have that $W^{<s>}$ is well-defined.
By  Corollary~\ref{cor!sumofGafornevenwere},   we have
\[
    W =  W^{<h>}  +\sum_{s=2}^{h/2}  W^{<s>}.
\]

For simplicity, we set $B = U^{<h>}$.
Assume $\; s \in \{ 2, 3,  \dots ,h/2 \}$. We have $  T_Q  (  W^{<s>}  ) = 0$, since 
the variable $a_{2s ,s}$ is an element of  $ \mathcal{A}_{Q,h-1} $ but
does not appear in $W^{<s>}$.  As a consequence, we have
\[
    T_Q(W) =   T_Q(W^{<h>}).
\]
Hence, using that   $Q_h = B_h$,
to prove Equation~(\ref {eqn!eqnforWandQ}) it is enough to prove that 
\[
           T_Q(W^{<h>}) = 1/ (B_h)^2.
\]
Taking into account  Remark~\ref{rem!aboutcharacterstic2fifferentialoperators}, it follows that
to prove Equation~(\ref {eqn!eqnforWandQ}) it is enough to prove that 
\begin{equation}  \label {eqn!eqnforUandQ}
       T_Q (  \prod_{i=2}^{h} B_i )    =     \prod_{i=2}^{h/2} ( B_i)^2.
\end{equation}
Using the definition of  $r_1$,  we get  that $B_i = P_i$  for all $1 \leq i \leq h-1$.
We set   
\[
      K =  B_h  - a_{2,1} P_1.
\]
By the well-known formula for the development
of  the determinant   $B_h$ using the row containing the element
$a_{2,1}$  it follows that the variable $a_{2,1}$ does not appear in $K$. 
Consequently,  the differential  operator  $\frac{\partial   } {\partial a_{2,1} }$ annihilates $K$. 
Since the same operator annihilates $P_j$ for all $1 \leq j \leq h$, 
and $ T_Q = T_P \circ \frac{\partial   } {\partial a_{2,1} }$,  we get
\[
       T_Q (  \prod_{i=2}^{h} B_i )  
           =  T_Q ( B_h \prod_{i=2}^{h-1} B_i )
            =  T_Q (  a_{2,1}  P_1 \prod_{i=2}^{h-1} P_i ) =                    
                  T_P (  P_1   \prod_{i=2}^{h-1} P_i ) 
               = \prod_{i=2}^{ h/2} (P_i)^2,
\]
with the last equality by the assumption for Part 1).  Since,
for $1 \leq i \leq h-1$, we have $B_i = P_i$, 
Equality~(\ref {eqn!eqnforUandQ}) follows, which finishes 
the proof of Part 1).  \\  \\

EXAMPLE (to help understand the above proof of Part 1):  If $ h =4$, we have
\[
      Q =  \left(   \begin{array}{cccc}
                        a_{2,1}    &    a_{2,2}    &    a_{2,3}  &  a_{2,4}    \\
                        a_{3,1}    &    a_{3,2}    &    a_{3,3}  &  a_{3,4}   \\
                        a_{4,1}    &    a_{4,2}    &    a_{4,3}  &  a_{4,4}   \\    
        \end{array}         \right),  \quad  \quad 
      B = U^{<h>} =  \left(   \begin{array}{cccc}
                        a_{2,1}    &    a_{2,2}    &    a_{2,3}  &  1    \\
                        a_{3,1}    &    a_{3,2}    &    a_{3,3}  &  0   \\
                        a_{4,1}    &    a_{4,2}    &    a_{4,3}  &  0   \\    
        \end{array}         \right), 
\]
\[
       P   =  \left(   \begin{array}{ccc}
                        a_{3,1}    &    a_{3,2}    &    a_{3,3}   \\
                        a_{4,1}    &    a_{4,2}    &    a_{4,3}      
        \end{array}         \right)
\]
and  
\[    
          T_Q =  \frac { \partial^3}{\partial a_{2,1} \; \partial a_{3,2} \; \partial  a_{4,2}} , 
               \quad \quad  T_P  =  \frac { \partial^2}{\partial a_{3,2} \; \partial  a_{4,2}}.
\]

\vspace{30pt}

  We now prove  Part 2) using similar arguments to the ones used
in the proof of Part 1).   We set 
\[
    W = \frac { \prod_{i=(h+2)/2}^{h} N_i } { N_{h+1}   \prod_{i=1}^{h/2} N_i }. 
\]
Using   Remark~\ref{rem!aboutcharacterstic2fifferentialoperators}, it is enough to
prove that
\begin{equation}  \label {eqn!eqnforWandQPart2}
              T_N (  W  )    =   1/ (N_{h+1})^2.
\end{equation}

Denote by $r_2$ the transpose of the $1 \times h$ matrix  
$\; (1,0,0, \dots , 0)$.
For $\; s \in \{ 1, 2,  \dots ,h/2 \} \cup \{h+1 \} \;$ 
 we denote by $X^{<s>}$ the $ h \times (h+1)$  matrix
obtained by replacing
the $s$-th column of $N$  with  $r_2$.

We set, for $s \in \{1,2, \dots , h/2 \} \cup \{ h+1 \}$, 
\[
     W^{<s>}  =    \frac { \prod_{j=(h+2)/2}^{h} X^{<s>}_j } {X^{<s>}_{h+1}   \prod_{j=1}^{h/2} X^{<s>}_j }. 
\]
Since $X^{<s>}_j \not= 0$, for all $1 \leq j \leq h + 1$,  we have that $W^{<s>}$ is well-defined.
By  Corollary~\ref{cor!sumofGafornoddwere},   we have
\[
    W =  W^{<h+1>}  +\sum_{s=1}^{h/2}  W^{<s>}.
\]

For simplicity, we set $C = X^{<h+1>}$.
Assume $\; s \in \{ 1, 2,  \dots ,h/2 \}$. We have $ T_N  (  W^{<s>}  ) = 0$, 
since  the variable $a_{2s ,s}$ is an element of  $ \mathcal{A}_{N,h} $ but
does not appear in $W^{<s>}$.  As a consequence, we have
\[
    T_N(W) =   T_N(W^{<h+1>}).
\]
Hence, using that   $N_{h+1} =  C_{h+1}$,
to prove Equation~(\ref {eqn!eqnforWandQPart2}) it is enough to prove that 
\[
           T_N(W^{<h+1>}) = 1/ ( C_{h+1})^2.
\]

Taking into account  Remark~\ref{rem!aboutcharacterstic2fifferentialoperators}, it follows that
to prove Equation~(\ref {eqn!eqnforWandQPart2}) it is enough to prove that 
\begin{equation}  \label {eqn!eqnforUandQPart2}
       T_N (  \prod_{i=1}^{h+1} C_i )    =     \prod_{i=1}^{h/2}  (C_i)^2.
\end{equation}
Using the definition of  $r_2$,  we get  that $ C_i = Q_i$  for all $1 \leq i \leq h$.
We set   
\[
               K =  C_{h+1}  - a_{1,1} Q_1.
\]
By the well-known formula for the development
of  the determinant   $ C_{h+1}$ using the row containing the element
$a_{1,1}$  it follows that the variable $a_{1,1}$ does not appear in $K$. 
Consequently,  the differential  operator  $\frac{\partial   } {\partial a_{1,1} }$ annihilates $K$. 
Since the same operator annihilates $Q_j$ for all $1 \leq j \leq h$, 
and $ T_N = T_Q \circ \frac{\partial   } {\partial a_{1,1} }$,  we get  
\begin {align*}
    T_N (  \prod_{i=1}^{h+1} C_i )  
         &  =  T_N ( C_{h+1}   \prod_{i=1}^{h} C_i )
            =  T_N (  a_{1,1}  Q_1 \prod_{i=1}^{h} Q_i ) =   T_N (  a_{1,1}  (Q_1)^2   \prod_{i=2}^{h} Q_i ) \\
        &   =  T_Q (  (Q_1)^2   \prod_{i=2}^{h} Q_i )  =
                    (Q_1)^2    \; T_Q (   \prod_{i=2}^{h} Q_i )  
               =  (Q_1)^2    \prod_{i=2}^{ h/2} (Q_i)^2
\end {align*}
with the last two equalities by Remark~\ref{rem!aboutcharacterstic2fifferentialoperators}
and the assumption for Part 2).  Since,
for $1 \leq i \leq h$, we have $ C_i = Q_i$, 
Equality~(\ref {eqn!eqnforUandQPart2}) follows, which finishes 
the proof of Part 2).  \\  \\

EXAMPLE (to help understand the above proof of Part 2):  If $h =4$, we have
\[
      N =  \left(   \begin{array}{ccccc}
                        a_{1,1}    &    a_{1,2}    &    a_{1,3}  &  a_{1,4}   &  a_{1,5}    \\
                        a_{2,1}    &    a_{2,2}    &    a_{2,3}  &  a_{2,4}   &  a_{2,5}    \\
                        a_{3,1}    &    a_{3,2}    &    a_{3,3}  &  a_{3,4}   &  a_{3,5}   \\
                        a_{4,1}    &    a_{4,2}    &    a_{4,3}  &  a_{4,4}   &  a_{4,5}   
        \end{array}         \right), \quad  \quad 
   C =   X^{<h+1>}  =  \left(   \begin{array}{ccccc}
                        a_{1,1}    &    a_{1,2}    &    a_{1,3}  &  a_{1,4}   &  1    \\
                        a_{2,1}    &    a_{2,2}    &    a_{2,3}  &  a_{2,4}   &   0    \\
                        a_{3,1}    &    a_{3,2}    &    a_{3,3}  &  a_{3,4}   &  0  \\
                        a_{4,1}    &    a_{4,2}    &    a_{4,3}  &  a_{4,4}   &   0  
        \end{array}         \right), \quad 
\]   
\[
      Q =  \left(   \begin{array}{cccc}
                        a_{2,1}    &    a_{2,2}    &    a_{2,3}  &  a_{2,4}    \\
                        a_{3,1}    &    a_{3,2}    &    a_{3,3}  &  a_{3,4}   \\
                        a_{4,1}    &    a_{4,2}    &    a_{4,3}  &  a_{4,4}   \\    
        \end{array}         \right)
\]
and  
\[    
          T_N =  \frac { \partial^4}{\partial a_{1,1} \; \partial a_{2,1} \; \partial a_{3,2} \; \partial  a_{4,2}} , 
               \quad \quad \quad    T_Q =  \frac { \partial^3}{\partial a_{2,1} \; \partial a_{3,2} \; \partial  a_{4,2}}.
\]
\end {proof}

\begin{theorem}
\label{theorem!computingdifferoperators} 
1)   Assume $h \geq 2$ is even.   We have   
\[
       T_{N,h}  (  \prod_{i=1}^{h+1} N^{(h)}_i )    =     \prod_{i=1}^{h/2} (N^{(h)}_i)^2.
\]

2)   Assume $h \geq 3$ is odd.  We have  
\[
       T_{Q,h} (  \prod_{i=2}^{h+1} Q^{(h)}_i )    =     \prod_{i=2}^{(h+1)/2} (Q^{(h)}_i)^2.
\] 
\end{theorem}

\begin{proof}   It is obvious that Part 2) is equivalent to the statement that
 for all even integers $h \geq 4$  we have
\[
       T_{Q,h-1} (  \prod_{i=2}^{h} Q^{(h-1)}_i )    =     \prod_{i=2}^{h/2} (Q^{(h-1)}_i)^2.
\]
Using induction on  the even integer $h \geq 2$, the proof of the present theorem 
follows by combining Proposition~\ref{prop!Prop2fordifferoperators},   
which provides the  starting 
case $h=2$,  and  
Propositions~\ref{prop!Prop1fordifferoperators}  and 
\ref{prop!Prop3fordifferoperators},  which provide  the inductive step.
\end{proof}

\begin{remark}  Conjecture~\ref{conj!generalisingthmcomputingdifferoperators} 
contains  a conjectural statement generalising Theorem~\ref{theorem!computingdifferoperators}.
\end{remark}

\section{Anisotropicity implies the Lefschetz Properties }  \label{sect!anisotropicityimplieswlp}

In the present section we investigate the relations 
between generic anisotropicity and the
Lefschetz properties. As an application, 
in Theorem~\ref{thm!secondproofofgconjecture} we give a second proof
of   McMullen's g-conjecture for simplicial spheres.

Assume  $k_1$ is a field of arbitrary characteristic, 
$n \geq 1$ is an integer,  and $D$ is a simplicial   sphere of dimension $n$  and
vertex set $\{1,2,  \dots , m  \}$.

We denote by $S(D)$ the suspension of $D$. More precisely, it is the simplicial complex 
with vertex set 
$\{1,2,  \dots  , m+2 \}$  and set of facets  equal to 
  \[ 
      \{  \sigma \cup  \{x_{m+1}\} : \sigma \in F(D) \}  \; \; \cup      \; \;
               \{  \sigma \cup \{x_{m+2}\} : \sigma \in F(D) \},
\]
where $F(D)$ denotes the set of facets of $D$.
It is well-known that $S(D)$ is a simplicial sphere of dimension  $n+1$. 
Moreover, we denote by $k$ the field of fractions of the polynomial
ring
\[
             k_1 [ a_{i,j} :  1 \leq i \leq n+2,  \;  1 \leq j \leq m+2 ].
\]
The proof of the following theorem will be given in Subsection~\ref{subsect!proofopropositionvoere}.

\begin{theorem} \label{thm!genericanisotropicityimplieswlpnonprecise}
Assume that  $S(D)$ is generically anisotropic over the field $k_1$.  Then
the graded $k$-algebra  $k[D]$ has the Weak Lefschetz  Property.    
\end{theorem}

All three statements in the following theorem are results originally due to Adiprasito \cite{A1,A2}.
The proof of the theorem will be given in Subsection~\ref{subs!proofofThmgconjecture}.

\begin{theorem} \label{thm!secondproofofgconjecture}
 (Adiprasito) Assume $D$ is a simplicial sphere of dimension $n$, with $n \geq 1$. Then

i)  McMullen's g-conjecture is true for $D$. 

ii)  Assume $k_1$ is an infinite field of characteristic $2$. Then the Stanley-Reisner
 ring $k_1[D]$ has  the Weak Lefschetz  Property.

iii)  Assume $k_1$ is an infinite field of characteristic $2$. Then the Stanley-Reisner
 ring $k_1[D]$ has  the Strong Lefschetz  Property.

\end{theorem}

\begin{remark} It is well-known that  iii) implies ii).  We state both ii) and iii), since 
in our
approach we first prove ii) and then use it to establish iii).  Notice also that the
paper \cite{A1} contains the stronger result that for any 
infinite field $k_1$ of arbitrary characteristic  the Stanley-Reisner
 ring $k_1[D]$ has  the Strong Lefschetz  Property.
\end{remark}

\subsection{The proof of Theorem~\ref{thm!genericanisotropicityimplieswlpnonprecise}} \label{subsect!proofopropositionvoere}

The aim of the present subsection is to prove 
Proposition~\ref{prop!genericanisotropicityimplieswlp}, since
it immediately implies  Theorem~\ref{thm!genericanisotropicityimplieswlpnonprecise}.
We use some key ideas and results of Swartz, 
which were developed in \cite[Section~4]{S1}.

We keep using the notations defined in Section~\ref{sect!anisotropicityimplieswlp}.
In addition, we set   $R_{sm} =k[x_1, \dots , x_{m}]$   and      $R = R_{sm} [ x_{m+1}, x_{m+2}]$. 
We denote by $I_D \subset R_{sm} $  the Stanley-Reisner ideal of $D$ over the field  $k$ and
by $I_{S(D)}  \subset  R$  the Stanley-Reisner ideal of $S(D)$ over the same field $k$.
It is clear that
\[
          I_{S(D)} = (I_{D}) + ( x_{m+1} x_{m+2}).
\]
We  denote  by  $k[D]=R_{sm}/I_D$  and   $k[S(D)]=R/I_{S(D)}$  the corresponding Stanley-Reisner rings over  $k$.

For $1 \leq i \leq  n+2$, we set 
\[
         f_i = \sum_{j=1}^{m+2} a_{i,j} x_j \in R.
\]
We use the notation $A= k[S(D)]/(f_1, \dots , f_{n+2})$, and
denote by $ \pi_A :  R \to A$ the natural projection $k$-algebra
homomorphism.    Therefore, 
$A$ is the generic Artinian reduction of $k_1[S(D)]$ in the sense of 
Definition~\ref{dfn!genericartinianreduction}.

We set  $J= I_{S(D)} : (x_{m+1}) \subset R$. In other words,
\[
      J = \{ u \in R  \;  : \;    u x_{m+1}  \in  I_{S(D)} \}.
\]
It is clear that $J = (I_D) + (x_{m+2})$.  
We use the notation 
\[
     B = \frac {R} {J+ (f_1, f_2, \dots , f_{n+2})} \; ,
\]
and we  denote by $ \pi_{B} :  R \to B$ the natural projection $k$-algebra
homomorphism.

For $2 \leq i  \leq n+2$ and $1 \leq j \leq m$, we set
\[
  c_{i,j} = \det  \left(   \begin{array}{cc}
                       a_{1,j}   &   a_{1,m+1}       \\
                       a_{i,j}    &  a_{i,m+1}
        \end{array}         \right)  \in k.
\]
In addition, for  $2 \leq i  \leq n+2$, we set  
\[
         g_i = \sum_{j=1}^{m} c_{i,j} x_j  \in R_{sm}. 
\]
Since, for all $2 \leq i \leq n+2$, it holds
\[
     g_i =  a_{i,m+1}f_1  -  a_{1,m+1}f_i  +  x_{m+2} 
                                (a_{1,m+1}a_{i,m+2}-a_{i,m+1} a_{1,m+2}),
\]
we get the following equality of ideals of $R$
\begin{equation}  \label{eqn!relatingfiandgi}
      (f_1, f_2,  \dots , f_{n+2}) + (x_{m+2}) = (f_1) + (g_2,g_3,  \dots ,g_{n+2}) + (x_{m+2}).
\end{equation}

We use the notation $C= k[D]/(g_2, g_3, \dots , g_{n+2})$, and
denote by $ \pi_C :  R_{sm} \to C$ the natural projection $k$-algebra
homomorphism.    We set 
\[
       \omega =  - \sum_{i=1}^{m} \frac { a_{1,i}}{ a_{1,m+1}}x_i \in R_{sm}.
\]
It is clear that $\pi_B(\omega) = \pi_B (x_{m+1})$. 
We consider the unique $k$-algebra homomorphism 
$\phi_{mod} :  R \to C$,  such that  $\phi_{mod} (x_i)=\pi_C  (x_i)$ for all $1 \leq i \leq m$,
$\phi_{mod}(x_{m+1})=\pi_C  (\omega)$   and $\phi_{mod}(x_{m+2})=0$.
From the definition of $\omega$  it follows that $\; f_1 \in \ker \phi_{mod}$. 
Hence,  Equation~(\ref{eqn!relatingfiandgi}) implies 
that the ideal  $\;J+ (f_1, \dots , f_{n+2}) \;$ is contained in the kernel of $\phi_{mod}$.
Consequently,  there exists an induced $k$-algebra homomorphism $ \phi : B \to C$
such that $\; \phi \circ \pi_B = \phi_{mod}$.

\begin{proposition} \label{prop!aboutBkanklink}
 The map $\phi$  is an isomorphism of graded $k$-algebras.
\end {proposition}

\begin {proof}
It is clear from the definition that $\phi$ preserves degrees.
   We consider the unique $k$-algebra homomorphism 
$R_{sm} \to B$, that sends $x_i$ to $\pi_B  (x_i)$, for all $1 \leq i \leq m$.
Using Equation~(\ref{eqn!relatingfiandgi}), it follows that 
the ideal $I_D +(g_2, g_3, \dots , g_{n+2})$  of $R_{sm}$ is inside its kernel, hence
there exists an induced $k$-algebra homomorphism $ \psi : C \to B$.
It follows from the definitions that $\psi$ is the inverse map of $\phi$.
\end {proof}

\begin{proposition} \label{prop!aboutABC}
  i)    The $k$-algebra $A$ is  graded, Artinian  and Gorenstein with socle degree equal to $n+2$.
 
   ii)    The $k$-algebras $B$ and $C$ are graded, Artinian and Gorenstein with socle degree equal to $n+1$.
\end {proposition}

\begin {proof}   We first remark that,  by Proposition~ \ref{prop!aboutBkanklink},
the graded $k$-algebras $B$ and $C$ are isomorphic.

By Remark~\ref{rem!basicpropertiesofA},
the $k$-algebra $k[S(D)]$ is graded and Gorenstein
with Krull dimension equal to $n+2$. Moreover, by the same remark 
$A$ is Artinian and Gorenstein with socle degree equal to  $n+2$.

Since $I_{S(D)} \subset J$, there exists a unique surjective homomorphism
of $k$-algebras $\pi_{new} : A \to B$, such that  $\pi_{new} \circ \pi_A =  \pi_B$.
Since $\pi_{new}$ is surjective and $A$ is Artinian, it follows 
that $B$ is Artinian.   Since $C$ is isomorphic to $B$ we get that $C$ is also Artinian.  
By Remark~\ref{rem!basicpropertiesofA},
the $k$-algebra $k[D]$ is graded and Gorenstein
with Krull dimension equal to $n+1$.  It follows that the sequence 
$g_2, \dots , g_{n+2}$ is a regular sequence for $k[D]$. 
This implies that the $k$-algebra $C$ is Gorenstein and, using again
Remark~\ref{rem!basicpropertiesofA}, that  
the  socle degree of $C$ is  equal to  $n+1$.
\end {proof}

We consider the homomorphism of  $R$-modules 
$R \to A$, that sends $u$ to $\pi_A  (x_{m+1}u)$, for all $u \in R$. 
It is clear that the ideal  $J+ (f_1, \dots , f_{n+2})$ of $R$ is inside its
kernel. Hence, we get an induced 
homomorphism of  $R$-modules $m_{x_{m+1}}  : B \to A$, such that
\[
              m_{x_{m+1}} ( \pi_B  (u))= \pi_A  (x_{m+1}u) 
\]
for all $u \in R$. The following proposition is a special case
of \cite[Proposition~4.24]{S1}.

\begin{proposition} \label{prop!aboutmultiplinjective} (Swartz)
    The homomorphism  $ m_{x_{m+1}}$ is injective.
\end {proposition}

\begin {proof}   
Recall the map  $\psi$ defined in the proof of Proposition~\ref{prop!aboutBkanklink}.
We set
\[
     \delta =  m_{x_{m+1}} \circ \psi  \; :  \; C  \to  A.
\] 
Since $\psi$ is an isomorphism, it is enough to prove that $\delta$ is injective.

Since, for all $j \geq 0$,  we have  $\delta (C_j) \subset A_{j+1}$,
to prove that $\delta$    is injective 
it is enough to assume  that $0 \leq j \leq n+1$ and $u \in (R_{sm})_j$
is a homogeneous element   of degree $j$ 
such that    $\pi_C(u)  \not= 0$, and prove that 
$ \delta (\pi_C ( u )) \not= 0$.   In order to get a contradiction,
we assume that 
\begin {equation}  \label{eqn!forcontradictionuno1}
           \delta (\pi_C ( u ))  = 0.
\end{equation}
 
By Proposition~\ref{prop!aboutABC},
  $C$ is a graded Artinian Gorenstein $k$-algebra with socle degree $n+1$. 
Therefore,   by Remark~\ref{rem!poincare_duality_for_gor},   there exists
$w \in  (R_{sm})_{n+1-j}\;$ such that $ \pi_C(u w) \not= 0$.  Using  Equation~(\ref{eqn!forcontradictionuno1})
\begin {equation}  \label{eqn!forcontradictionuno2}
           \delta (\pi_C ( u w))  =  \pi_A (x_{m+1}uw ) =\pi_A (x_{m+1}u)  \pi_A(w ) = 
                             \delta  (\pi_C ( u ))   \pi_A(w )  = 0.
\end{equation}

We fix a facet $\{a_1, \dots , a_{n+1} \}$ of $D$
and consider the facet  $\{a_1, \dots , a_{n+1}, m+1 \}$ of $S(D)$.
We set
\[
      z_C =   \prod_{r=1}^{n+1} x_{a_r}  \in R_{sm}, \quad  \quad 
      z_A =  x_{m+1} \prod_{r=1}^{n+1} x_{a_r}  \in R.
\]
Using the discussion after the proof of  Corollary~\ref{cor!abouttwofacets},
  $\pi_A (z_A)$ is a nonzero element of 
$A_{n+2}$.    
By the same discussion, $\pi_C (z_C)$ is nonzero, hence is a basis 
of the $1$-dimensional
$k$-vector space $C_{n+1}$.   Therefore,  there exists a nonzero element 
$\lambda \in k$ such that   
\[
   \pi_C(u w)   = \lambda \pi_C (z_C).
\]
Consequently, 
\[
    \delta ( \pi_C(u w))   =  \delta (\lambda \pi_C (z_C)) = \lambda  \delta (\pi_C (z_C)) = 
              \lambda    \pi_A (x_{m+1} z_C)                 =   \lambda    \pi_A (z_A)      \not= 0,
\]
which contradicts Equation~(\ref{eqn!forcontradictionuno2}).
\end {proof}

The following corollary follows immediately from Proposition~\ref{prop!aboutmultiplinjective}.

\begin{corollary} \label{cor!aboutnonzeroelementinB}
    Assume $u \in R$.  Then the following are equivalent

  i)    We have   $\;  \pi_{B}  (u) = 0$.
 
   ii)   We have   $\;  \pi_{A}  (x_{m+1}u ) = 0$.
\end {corollary}

\begin{proposition} \label{prop!genericanisotropicityimplieswlp}
 Assume $S(D)$ is generically anisotropic over the field $k_1$. Then the element
 $\pi_C (\omega)$
is a Weak Lefschetz element for the Artinian $k$-algebra
$C$. As a consequence, the graded $k$-algebra 
$k[D]$ has the Weak Lefschetz Property.    
\end{proposition}

\begin {proof}   We denote by $p$ the integral value of the rational number $n/2$.

Using that  $\pi_B(\omega) = \pi_B (x_{m+1})$ and  Proposition~\ref{prop!aboutBkanklink}, it is enough to prove that
the element  $\pi_B (x_{m+1})$ 
is a Weak Lefschetz element for the Artinian $k$-algebra $B$.
By Proposition~\ref{prop!aboutABC},
  $B$ is a graded Artinian Gorenstein $k$-algebra with socle degree $n+1$.  Using
\cite[Remark~2.4]{MZ},  it is enough to prove that the multiplication by
$\pi_B (x_{m+1})$ map from $B_p$ to $B_{p+1}$ is injective.

Assume $u \in R_p$ has the property  
\[
           \pi_B ( x_{m+1}u )   =  0.
\]
 Using Corollary~\ref{cor!aboutnonzeroelementinB},
we have  
$\;  \pi_A ( x_{m+1}^2 u )   =  0$, hence
$\; \pi_A ( x_{m+1}^2 u^2 )   =  0$.
Using  that the socle degree of $A$ is $n+2$ and  the assumption
that $S(D)$ is generically anisotropic over the field $k_1$, we get 
$\; \pi_A ( x_{m+1} u )  =  0$. 
Corollary~\ref{cor!aboutnonzeroelementinB} implies that  $\pi_B ( u )   =  0$.
\end {proof}

\begin{proposition} \label{prop!genericanisotropicityandnevenimplyslp}
 Assume the dimension  of $D$ is even and 
$S(D)$ is generically anisotropic over the field $k_1$. 
Then the element   $\pi_C (\omega)$
is a Strong Lefschetz element for the Artinian $k$-algebra
$C$. As a consequence, 
the  graded  $k$-algebra $k[D]$ has the Strong Lefschetz Property.    
\end{proposition}

\begin {proof}   
We set $z=\pi_B (\omega)$.
Using Proposition~\ref{prop!aboutBkanklink}, it is enough to prove that
the element  $z$ 
is a Strong Lefschetz element for the Artinian $k$-algebra $B$.
By Proposition~\ref{prop!aboutABC},
$B$ is a graded Artinian Gorenstein $k$-algebra with socle degree $n+1$.  
Hence, to finish the proof it is enough to 
prove that, for all $i$ with $0 \leq 2i \leq n+1$,
the multiplication by $z^{n+1-2i}$ map $B_i \to B_{n+1-i}$ is injective.

Assume $0 \leq 2i \leq n+1$ and $u \in R_i$ has the property
\[
           z^{n+1-2i} \pi_B (u)  = 0.
\]
Using that  $\;z=\pi_B (x_{m+1}) \;$
and Corollary~\ref{cor!aboutnonzeroelementinB},
we get   $\;\pi_A ( x_{m+1}^{n+2-2i} u )   =  0$, which
implies that $\;\pi_A ( x_{m+1}^{n+2-2i} u^2 )   =  0$.

Since $n$ is even, the socle degree of $A$ is $n+2$ and we assumed
that $S(D)$ is generically anisotropic over the field $k_1$, we get 
$\; \pi_A ( x_{m+1}^{(n+2)/2-i} u )   =  0$. 
Corollary~\ref{cor!aboutnonzeroelementinB} implies that 
 $\pi_B ( x_{m+1}^{(n+2)/2-i-1} u )   =  0$,  therefore
\begin{equation}  \label{eqn!ecsececs}
        z^{(n+2)/2-i-1} \;  \pi_B (u)   =  0.
\end{equation}

By the proof of  Proposition~\ref{prop!genericanisotropicityimplieswlp},
$z$ is a  Weak Lefschetz element for $B$. 
Hence, the multiplication by $z$ map $B_{n/2} \to B_{n/2+1}$
is injective.  
Using   Proposition~\ref{prop!injectivityforsmallerdegerees},
we have that, for all $t$ with 
$0 \leq t \leq n/2$,  the multiplication by $z$ map 
$B_{t} \to B_{t+1}$ is injective.  Consequently,
Equation~(\ref{eqn!ecsececs}) implies that  $\pi_B ( u )   =  0$.
\end {proof}

\subsection{The proof of Theorem~\ref{thm!secondproofofgconjecture}} \label{subs!proofofThmgconjecture}
We start the proof of Theorem~\ref{thm!secondproofofgconjecture}.
We denote by $k_{mod}$ the  field of fractions of the polynomial
ring
\[
             k_1 [ a_{i,j} :  1 \leq i \leq n+1,  \;  1 \leq j \leq m ].
\]
For $1 \leq i \leq  n+1$, we set 
\[
         f_{mod,i} = \sum_{j=1}^{m} a_{i,j} x_j.
\]
We use the notation 
\[ 
      A_{mod}= k_{mod}[D]/(f_{mod,1}, \dots , f_{mod,n+1}).
\]
Hence,
 $A_{mod}$ is the generic Artinian reduction of $k_1[D]$ in the sense of 
Definition~\ref{dfn!genericartinianreduction}.

We first prove Part i). We denote by  $k_1 $ the field $\mathbb{Z}/(2)$ with two elements.
By Theorem~\ref{thm!anisotropicityinchar2},   $S(D)$ is generically anisotropic over the field $k_1$.  Hence,
by   Theorem~\ref{thm!genericanisotropicityimplieswlpnonprecise}, $k[D]$ has
 the Weak Lefschetz  Property. It is well-known  (\cite{St3}) 
 that this implies 
that McMullen's g-conjecture is true for $D$.

We now prove Part ii). Assume $k_1$ is an infinite field of characteristic $2$. 
By Theorem~\ref{thm!anisotropicityinchar2},   $S(D)$ is generically anisotropic over the field $k_1$.  Hence,
by   Theorem~\ref{thm!genericanisotropicityimplieswlpnonprecise}, $k[D]$ has
 the Weak Lefschetz  Property.  Using   Proposition~\ref{prop!aboutAvsAbig},
 $k_1[D]$ also has  the Weak Lefschetz  Property.

We now prove Part iii). Assume $k_1$ is an infinite field of characteristic $2$. 
By Theorem~\ref{thm!anisotropicityinchar2},   $S(D)$ is generically anisotropic over the field $k_1$.
If the dimension $n$ of $D$ is even, Proposition~\ref{prop!genericanisotropicityandnevenimplyslp}
implies that $k[D]$ has  the Strong Lefschetz  Property.  Using   Proposition~\ref{prop!aboutAvsAbigforSLP},
 $k_1[D]$ also has  the Strong Lefschetz  Property.

Assume now that $n$ is odd.  
By Part ii),  $k_1[D]$ has  the Weak Lefschetz  Property.
Using  Proposition~\ref{prop!threeequivalencesforWLP}, the Artinian   
$k_{mod}$-algebra 
$A_{mod}$  has the  Weak Lefschetz  Property. 
By Theorem~\ref{thm!anisotropicityinchar2},	
$D$ is generically anisotropic over the field $k_1$.
Hence,  for all $i$ with  $0 \leq i \leq (n+1)/2$ and  all
$0 \not= u \in (A_{mod})_i$,  
we  have $u^2 \not=0 $.
Proposition~\ref{prop!inevensocledegreeslpandanisotimplyslp} now
implies that $A_{mod}$ has the  Strong Lefschetz  Property.  
Since $k_1$ is infinite, 
Proposition~\ref{prop!threeequivalencesforSLP} implies
that  $k_1[D]$  has  the Strong Lefschetz  Property.  
This finishes the proof of Theorem~\ref{thm!secondproofofgconjecture}.

\begin {corollary} \label{cor!AmodhasSLP}  
Assume $D$ is a simplicial sphere of dimension $n \geq 1$,
and $k_1$ is a  (finite or infinite) field of characteristic $2$. 
Then the $k_{mod}$-algebra  $A_{mod}$ has the Strong Lefschetz Property.
\end {corollary}

\begin {proof}  The field $k_{mod}$ is infinite and has
characteristic $2$. 
Hence,  Theorem~\ref{thm!secondproofofgconjecture} implies that
the $k_{mod}$-algebra  $k_{mod}[D]$ has the Strong Lefschetz Property.
Using Proposition~\ref{prop!threeequivalencesforSLP}, the result follows.
\end{proof}

\section{Anisotropicity in dimension $1$} 
\label {sec!1dimensional}

In this section $k_1$ denotes a field of arbitrary characteristic.

We assume  that  $m \geq 3$ and $D$ is the boundary of the $m$-gon
with vertex set $\{1, \dots , m\}$.
We also assume the following: the  vertex $1$ is connected
to the vertices $m$ and $2$,   the  vertex $i$ is connected
to the vertices $i-1$ and $i+1$ when $2 \leq i \leq m-1$, and the
vertex  $m$ is connected
to the vertices $m-1$ and $1$.  

We denote by  $S_{sp}$ the polynomial
ring
\[
          S_{sp} = k_1 [ a_{i,j} :  1 \leq i \leq 2,  \;  1 \leq j \leq m ]
\]
and by  $k$ the field of fractions of $S_{sp}$.
We define the polynomial ring  $R = k[x_1, \dots , x_m]$. 
We denote by $I_D \subset R$ the
Stanley-Reisner ideal of $D$, and we set $k[D]=R/I_D$.  
For $1 \leq i \leq  2$, we set 
\[
         f_i = \sum_{j=1}^m a_{i,j} x_j,
\]
and we define $A = k[D]/(f_1,  f_{2})$.  
Therefore, $A$ is the generic Artinian reduction of $k_1[D]$ in the sense of 
Definition~\ref{dfn!genericartinianreduction}.

If  $m \geq 4$ we have
\[
   I_D = (x_1x_j : 3 \leq j \leq m-1) +  (x_ix_j : 2 \leq i  \leq m-2 , i+2 \leq j \leq m),
\]
while if $m=3$, we have $I_D = (x_1x_2x_3)$.

We fix the ordered facet $(1,2)$  of $D$. Following
Equations~(\ref{eqn!dfnofcapitalpsi}) and (\ref{eqn!dfnofrhoe}), 
we set
\[       
          \Psi = \Psi_{(1,2)}  :  A_{2} \to k \quad  \quad \text{and} \quad   \quad \rho = \rho_{(1,2)} : A_{1} \times A_{1} \to k.
\]

\begin{proposition}
\label{prop!psiinth1dimcase} 
For $1 \leq i \leq m-1$, we have
\[ 
     ( \Psi \circ  \pi ) (x_ix_{i+1})  = \frac {1} {[i, i+1]}.
\]
Moreover, we have
\[ 
     ( \Psi \circ  \pi ) (x_1x_{m}) = \frac {1} {[m, 1]}, \; \;
           ( \Psi \circ  \pi ) (x_1^2)  = -\frac {[m,2]} {[m, 1][1,2]}, \; \;
            ( \Psi \circ  \pi )   (x_m^2)  = -\frac {[m-1,1]} {[m-1, m][m,1]} 
\]
and
\[
           ( \Psi \circ  \pi )  (x_i^2)  = -\frac {[i-1,i+1]} {[i-1, i][i,i+1]} 
\]
for $2 \leq i \leq m-1$.
\end{proposition}

\begin{proof}
Combining Proposition~\ref{prop!reducedEquation} 
with Proposition~\ref{prop!oneSquareequation}, the result is immediate.
\end{proof}

\begin{proposition}
\label{prop!basisinth1dimcase} 
We have $\; \dim_k A_1 = m-2$. If $\mathcal{S}$ is any 
subset of $\{1, \dots , m\}$ of cardinality $m-2$, 
then the set $ \; \{  \pi(x_i)  :  i \in \mathcal{S}  \} \;$
is a $k$-basis of $A_1$.
\end{proposition}

\begin{proof} 
We denote by $M$  the $2 \times m$
matrix with $(i,j)$-entry equal to  $a_{i,j}$.
The determinant of every
$2 \times 2$ submatrix of $M$ is a nonzero
element of the field $k$.  Since $A= k [D]/(f_1,f_2)$ and $I_D$ is 
a homogeneous ideal with generators of  degrees
$\geq 2$,  the result follows.
\end{proof}

For $1 \leq i \leq m-2$,  we set $e_i = \pi(x_{i+1})$.
By Proposition~\ref{prop!basisinth1dimcase}, the finite sequence     
\[
    e_1 ,     e_2 , \dots ,     e_{m-2} 
\]
is an ordered basis of $A_1$.  We denote by $N_m$ the 
$(m-2) \times  (m-2) $ symmetric matrix, with $(i,j)$-entry
equal to $\rho (e_i, e_j)$. We call $N_m$ the matrix of $\rho$
with respect to the ordered basis.

\begin{remark} 
 \label{rem!PlueckerRelationsin1dim}
  Assume  $a,b,c,d  \in \{1, \dots ,m \}$.  
  Then, we have the  well-known  Pl\"{u}cker identity
  \[
        [a,b][c,d]-[a,c][b,d] + [a,d][b,c] =0,
 \]
see~\cite[Theorem~5.2.3]{LB}.
\end{remark}

\begin{proposition}
\label{prop!detofNm1dimcase} 
We have 
\[
        \det (N_m)  =  (-1)^m \frac{[1,m]}{\prod_{i=1}^{m-1}[i,i+1]}.
\]
\end{proposition}

\begin{proof}  We use induction on $m \geq 3$. For $m =3$, it follows from Proposition~\ref{prop!psiinth1dimcase}. 

Assume $m = 4$. Then we have to
compute the determinant of the matrix 
\[
       N_4 =  \left(   \begin{array}{cc}
                        -\frac{[1,3]}{[1,2][2,3]}    &   \frac{1}{[2,3]}      \\
                        \frac{1}{[2,3]}    &    -\frac{[2,4]}{[2,3][3,4]}    
        \end{array}         \right).
\]
It is equal to 
\[
                       \frac{[1,3]}{[1,2][2,3]} \frac{[2,4]}{[2,3][3,4]}    -
                                \big(   \frac{1}{[2,3]} \big)^2 =
                       \frac{[1,3][2,4] -[1,2][3,4]}{[1,2]|[2,3]^2[3,4] }. 
\]
Using the Pl\"{u}cker identity $[1,2][3,4] - [1,3] [2,4]  + [1,4][2,3]  = 0$ 
(see Remark~\ref{rem!PlueckerRelationsin1dim}) the result follows.

Assume now $m \geq 5$ and that the result holds for all
previous values up to $m-1$. Using Proposition~\ref{prop!psiinth1dimcase},
we have that  $N_m$ has the block format
\[
       N_m =  \left(   \begin{array}{cc}
                       N_{m-1}   &   v^t       \\
                       v    &    -\frac{[m-2,m]}{[m-2,m-1][m-1,m]}    
        \end{array}         \right),
\]
where $v$ is  the $ (m-3) \times 1 $ matrix
\[
       v =  \left(   \begin{array}{ccccc}
                       0 &  0   &  \dots    &  0 &   \frac{1}{[m-2,m-1]} 
        \end{array}         \right).
\]
Morever, a similar block format statement holds for 
the matrix  $N_{m-1}$.

Developing the determinant of  $N_m$ using the
last column, and using the inductive hypothesis together
with the  Pl\"{u}cker identity (see Remark~\ref{rem!PlueckerRelationsin1dim})
\[
      [1,m-2][m-1,m]-[1,m-1][m-2,m] + [1,m][m-2,m-1] =0,   
\]
 we get
\[
        \det (N_m)  =     -\frac{[m-2,m]}{[m-2,m-1][m-1,m]} \det (N_{m-1}) 
                            -    (\frac{1}{[m-2,m-1]})^2 \det (N_{m-2}) 
\]
\[
 =    (-1)^{m-1}  \frac{-[m-2,m][1,m-1]}{[m-2,m-1][m-1,m]
                      \prod_{i=1}^{m-2}[i,i+1]} 
                          - (-1)^{m-2} \frac{[1,m-2]}{[m-2,m-1]^2                   
                                  \prod_{i=1}^{m-3}[i,i+1]} 
\]
\[
 =                     (-1)^{m-2} \big(  \frac{[1,m-1][m-2,m]}{[m-2,m-1] 
                                \prod_{i=1}^{m-1}[i,i+1]} 
                          -   \frac{[1,m-2]}{[m-2,m-1]
                              \prod_{i=1}^{m-2}[i,i+1]} \big)   
\]
\[
 =                     (-1)^{m-2} \big( \frac{[1,m-1][m-2,m] - [1,m-2][m-1,m]}
                              {[m-2,m-1] 
                                \prod_{i=1}^{m-1}[i,i+1]}  \big)   
\]
\[
 =                     (-1)^{m} \big( \frac{[1,m]}
                              { \prod_{i=1}^{m-1}[i,i+1]}  \big).    
\]   \end{proof}

\begin{remark} \label{rem!wellknownfactaboutdet}
Assume $1 \leq c < d \leq m$. It is well-known
that $[c,d]$ is an irreducible element of $S_{sp}$.
Hence, there exists an induced valuation map
\[  
  \val_{[c,d]}  :  k \setminus \{ 0 \}  \to \mathbb{Z}.
\]
Recall that  if $f,g \in S_{sp} \setminus \{ 0 \}$, 
then   $\val_{[c,d]} (f)$ is the largest integer $s$ such that
$[c,d]^s$ divides $f$  in $S_{sp}$, and
\[
     \val_{[c,d]}(f/g) = \val_{[c,d]} (f) - \val_{[c,d]} (g).
\]

\end{remark}

\begin{remark}
\label{rem!aboutdetindifferentbases}
Assume that $h$ is 
any ordered basis of $A_1$. We denote by 
$H$ the matrix of $\rho$
with respect to $h$.  By the basic theory of
bilinear forms,  there exists 
an invertible matrix $P$ with entries in $k$ such that
\[
    H =  P^t N_m P.
\]
As a consequence, using Proposition~\ref{prop!detofNm1dimcase},
\[ 
        \det (H) = 
          (-1)^m (\det P)^2  \frac{[1,m]}{\prod_{i=1}^{m-1}[i,i+1]}.
\]
Taking into account    Remark~\ref{rem!wellknownfactaboutdet},
we conclude that we can recover the simplicial
complex $D$ from  (the determinant of)
 $\rho$, since the set of facets of $D$ is exactly
the set of ordered pairs $(c,d)$ such that 
$1 \leq c < d \leq m$ and $\; \val_{[c,d]} (\det (H)) \;$ is 
an odd integer.    An interesting question is whether this 
holds for all simplicial spheres of odd dimension. In other words,
assume $E$ is a simplicial sphere of odd dimension $\geq 3$
and $e$ is an ordered facet of $E$.   Is it 
possible to recover $E$ from  (the determinant of)  the
symmetric bilinear form $\rho_e$?
\end {remark}

The proof of the following theorem will be given in Subsection~\ref{subs!pfofthm1dmanisotropic}.

\begin{theorem}
\label{thm!1dimisanisotropic} 
The simplicial sphere $D$ is generically anisotropic over $k_1$.
\end{theorem}

\subsection {The proof of Theorem~\ref{thm!1dimisanisotropic} } 
\label {subs!pfofthm1dmanisotropic}

We keep using the notations of Section~\ref{sec!1dimensional}.
Using  Remark~\ref{rem!relnbwithanisotropicity}, to prove 
 Theorem~\ref{thm!1dimisanisotropic}   it is enough to prove that the
 symmetric bilinear form $\rho : A_1 \times A_1 \to k$ is anisotropic.

We define a second basis of $A_1$, by using the
Gram-Schmidt orthogonalization. We set 
$ \tilde{e}_1  = e_1$, and we  inductively define
 \[
        \tilde{e}_i  = e_i + \frac{[1, i]}{[1, i+1]}  \tilde{e}_{i-1},
\]
for $2 \leq i \leq m-2$.

\begin{proposition}
\label{prop!computingetilde}
For all $1 \leq i \leq m-2$,  we have  
\begin{equation}  \label{eqn!computingetilde}
    \tilde{e}_i =   \sum_{t=2}^{i+1} \frac { [1,t] }{ [1,i+1]}  \pi(x_t).
\end{equation}
\end {proposition}

\begin{proof} 
We   prove Equation~(\ref{eqn!computingetilde}) using induction on $i$.
For $i=1$ it is true  by the definition of $\tilde{e}_1$.
Assume $1 \leq i \leq m-3$ and that  Equation~(\ref{eqn!computingetilde})  is true
for the value $i$. We have
\begin{align*}
    \tilde{e}_{i+1}  &=  e_{i+1} + \frac{[1,i+1]}{[1,i+2]}  \tilde{e}_{i}  =
         \pi(x_{i+2}) + \frac{[1,i+1]}{[1,i+2]} (  \sum_{t=2}^{i+1} \frac{[1,t]}{[1,i+1]} \pi(x_t)  )  \\
        & =  \pi(x_{i+2}) +  \sum_{t=2}^{i+1} \frac{[1,t]}{[1,i+2]} \pi(x_t)   = 
            \sum_{t=2}^{i+2} \frac{[1,t]}{[1,i+2]} \pi(x_t).
\end{align*}

\end{proof}

\begin{proposition}
\label{prop!rhointhepeprimebasis} 
For all $1 \leq i \leq m-2$, we have
\[ 
    \rho ( \tilde{e}_i ,  \tilde{e}_i ) = -\frac {[1,i+2]} {[1, i+1][i+1,i+2]}.
\]
Moreover, if  $1 \leq j \leq m-2$ and $j \not= i$,   we have
\[ 
     \rho  (  \tilde{e}_i ,  \tilde{e}_j ) =0.
\]
\end{proposition}

\begin{proof}   Assume $1 \leq i \leq m-2$.
We set $u= \sum_{t=2}^{i+1} [1,t] \pi(x_t)$.  
By Proposition~\ref{prop!aboutGausselim},
\[  
            \sum_{t=2}^{m} [1,t]  \pi(x_t) = 0.
\]     
Hence, if $1 \leq r \leq  i$, taking into account that $\pi(x_r x_t ) = 0$ 
when $r+2 \leq t \leq m$, we get 
\begin{equation}  \label{eqn!uannihilatesxr}
              u \; \pi(x_r) = 0.
\end{equation}   

Assume  $1 \leq j < i$.  Using Proposition~\ref{prop!computingetilde}, 
Equation~(\ref{eqn!uannihilatesxr}) implies that    $\rho  (  \tilde{e}_i ,  \tilde{e}_j ) =0$. 
Moreover,   Equation~(\ref{eqn!uannihilatesxr}) also implies that 
\begin{align*}
    \Psi ( u^2 ) &=  \Psi \big(  u \; (\sum_{t=2}^{i+1} [1,t] \pi(x_t)) \big)  =  \Psi \big(  [1,i+1] u \pi(x_{i+1})   \big)  \\
        & = 
               \Psi  \big( [1,i+1]   [1,i]    \pi(x_ix_{i+1})  + [1,i+1]^2 \pi(x_{i+1}^2) \big)   \\
     &=  [1,i+1] \big(\frac { [1,i] } { [i,i+1]} -  \frac { [1,i+1][i, i+2] } { [i,i+1][i+1,i+2]} \big)  \\
      &=      [1,i+1]  \frac { [1,i] [i+1,i+2] - [1,i+1][i, i+2] } { [i,i+1][i+1,i+2]}  \\
      &=
            -  [1,i+1]  \frac { [1,i+2] [i,i+1]} { [i,i+1][i+1,i+2]} =
            -   [1,i+1] \frac { [1,i+2] } {[i+1,i+2]},
\end{align*}
where we used Proposition~\ref{prop!psiinth1dimcase} and Remark~\ref{rem!PlueckerRelationsin1dim}.
Since $\tilde{e}_i = u/[1,i+1]$, this proves the formula for   $\rho ( \tilde{e}_i ,  \tilde{e}_i )$.
\end{proof}

We set
\[
      L = \prod_{s=2}^{m-1}  [1,s]  [s,s+1] 
\]
and, for $1 \leq t \leq m-2$,  we define  
$L_t = L/([1,t+1][t+1,t+2]) \in  S_{sp}$.

Using Proposition~\ref{prop!rhointhepeprimebasis}, it is clear 
that to prove Theorem~\ref{thm!1dimisanisotropic}
it is enough to prove that, if $d_t \in k$ satisfy
\[
     \sum_{t=1}^{m-2} d_t^2 \frac {[1,t+2]} {[1, t+1][t+1,t+2]} = 0,
\]
we then have $d_t = 0$ for all $1 \leq t \leq m-2$. By clearing denominators, it
is enough to prove the following proposition.

\begin{proposition}
\label{prop!vanishingofdi}
Assume $d_1, \dots , d_{m-2} \in S_{sp}$ satisfy
\begin{equation}  \label{eqn!eqnforci}
         \sum_{t=1}^{m-2}  d_t^2  [1,t+2]L_t   = 0. 
\end{equation}
Then, we have  $d_t = 0$, for all $1 \leq t \leq m-2$.
\end{proposition}

\begin{proof}  
We give to the polynomial ring $S_{sp}$ the lexicographic
ordering $>$ with
\[
    a_{1,1} >  a_{1,2} > \dots > a_{1,m} >    a_{2,1} >   a_{2,2} > \dots > a_{2,m}.
\]
Using  Corollary~\ref{cor!aboutinittermofsums}, it is enough to
prove that if  $\; i,j \;$ have the properties  $1 \leq i < j  \leq m-2$,
$d_i \not= 0$ and  $d_j \not= 0$, we then have 
\begin{equation} \label{eqn!leadtermsdifferent}
       \interm_{>} ( d_i^2  [1,i+2] L_i ) \not=   \interm_{>} ( d_j^2   [1,j+2] L_j ).
\end{equation}

Using the definitions of $L_i$ and $L_j$ and 
Remark~\ref{rem!aboutinittermnpo1}, we have
\[
       \interm_{>} ( d_i^2  [1,i+2] L_i ) =
         (\interm_{>} ( d_i))^2 \cdot (a_{1,1})^{m-2}  \cdot 
         \prod_{s=1}^{i} a_{1,s}  \cdot \prod_{s=i+2}^{m-1} a_{1,s} \cdot Q_i,
\]
and  
\[
       \interm_{>} ( d_j^2  [1,j+2] L_j ) =
         (\interm_{>} ( d_j))^2 \cdot (a_{1,1})^{m-2}  \cdot 
         \prod_{s=1}^{j} a_{1,s}  \cdot \prod_{s=j+2}^{m-1} a_{1,s} \cdot Q_j,
\]
where   $Q_i$ and  $Q_j$  are monomials in the variables $a_{2,1}, \dots , a_{2,m}$.
Therefore,  the variable $a_{1,j+1}$ appears in the monomial        $\; \interm_{>} ( d_i^2  [1,i+2] L_i ) \;$
with an odd  power, and in the monomial $\;\interm_{>} ( d_j^2  [1,j+2] L_j ) \;$  with an even power.
Hence, Inequality~(\ref{eqn!leadtermsdifferent}) is true.
\end {proof}

\begin{example}  Assume  $m=6$.  Equation~(\ref{eqn!eqnforci}) becomes
\[
       d_1^2[1,3]L_1 +   d_2^2[1,4]L_2 +   d_3^2[1,5]L_3 +  d_4^2[1,6]L_4  = 0,
\]
where
\[
     L_1 = \frac {L}{[1,2][2,3]}, \quad L_2 = \frac {L}{[1,3][3,4]},
         \quad L_3 = \frac {L}{[1,4][4,5]},   \quad L_4 = \frac {L}{[1,5][5,6]}
\]
and
\[
    L =  [1,2]  [1,3]  [1,4] [1,5] [2,3][3,4][4,5][5,6]. 
\]
\end{example}

\section{A general proposition related  to elimination}

In this section we describe a specific form of Gauss elimination that is used
in the present paper.

  Assume  $R$ is  a commutative ring with unit, and $n,m, Z$
are positive integers with $n < m  \leq Z$.   Assume  
that, for $ 1 \leq j \leq m$, $x_j$ is an elements of  $R$ 
and that for    $1 \leq i \leq n$ and $ 1 \leq j \leq Z$, 
$a_{i,j}$ is an element  of  $R$. 
We denote by $M$ the $n \times Z$ matrix with $(i,j)$-entry equal to $a_{i,j}$.

Assume $b_1, \dots , b_n$ are $n$ integers, with $1 \leq b_i \leq Z$,
for all $i$.    We denote by $[b_1, \dots , b_n]$ the determinant of
the $n \times n$ matrix, whose $i$-th column is equal to the $b_i$-th
column of $M$.  For $1 \leq i \leq n$, we set 
\[
       f_i  = \sum_{t = 1}^m  a_{i,t} x_t,
\]
and we denote by $I = (f_1, \dots, f_n)$ the ideal of $R$ generated
by the $f_i$.

\begin{proposition}
\label{prop!aboutGausselim} 
Assume $c_1, \dots , c_{n-1}$ are integers, with $1 \leq c_i \leq Z$ for all $i$. 
We  have
\[
    \sum_{t=1}^m  [c_1, c_2, \dots , c_{n-1},t] x_t  \in I.
\]
\end{proposition}

\begin{proof} 
   Denote by $N$ the $n \times (n-1)$ matrix, whose $i$-th column is
the $c_i$-th column of $M$. For $1 \leq j \leq n$, we denote by 
$N_j$ the determinant of the submatrix of $N$ obtained by deleting the 
$j$-th row of $N$. We claim  that
\[
   \sum_{t=1}^m  [c_1, c_2, \dots , c_{n-1},t] x_t   =
                     \sum_{j=1}^n (-1)^{j+n} N_j f_j.
\]
Indeed, on the left hand side, the coefficient of  $x_t$ is
$ [c_1, c_2, \dots , c_{n-1},t]$, while on the right hand side
the coefficient is equal to $\sum_{j=1}^n (-1)^{j+n} N_j a_{j,t}$.
The two quantities are equal, by developing the
determinant $ [c_1, c_2, \dots , c_{n-1},t]$ using the last column.
\end {proof}

\section{A general technique for proving a  polynomial 
                 is nonzero} 

Here we discuss a well-known general  method which is useful for proving that 
certain sums of products of bracket polynomials are nonzero. We use it in
the proof of  Proposition~\ref{prop!vanishingofdi}.

Assume $m \geq 1$, $k$ is a field and $R = k [x_{i} : 1 \leq i \leq m]$. We denote by
$\mathcal{A}_R$ the set of all monomials of $R$. In other words,
\[
   \mathcal{A}_R  = \{  x_1^{a_1} \cdots x_n^{a_m}  :   a_i \geq 0  \text  { for all } i   \}.
\]
Following \cite [Section~15.2]{Ei}, a monomial order on $R$ is a total order $>$ on    $\mathcal{A}_R$
such that if $u_1, u_2, w \in \mathcal{A}_R$ with $u_1 > u_2$ and $w \not= 1$, we then
have $\;  wu_1 > wu_2 > u_2$.
In addition, by the same reference,  the lexicographic  order on $R$ with $x_1 > x_2 > \dots > x_m$ 
is the total order  $>$ on    $\mathcal{A}_R$ defined  by  
$x_1^{a_1} \cdots x_m^{a_m}  > x_1^{b_1} \cdots x_m^{b_m}$ if and only if 
$a_i > b_i$ for the first index $i$ such that $a_i \not= b_i$.  It is a monomial
order on $R$.

Assume now $>$ is a monomomial order on $R$. It induces the initial monomial map, 
$ \interm_{>} :  R \setminus \{ 0 \} \to   \mathcal{A}_R$,  defined as follows.
Assume $f \in R \setminus \{ 0 \}$. Then, there exist (unique) $s > 0$, 
$g_1, \dots , g_s \in  \mathcal{A}_R$  and
$\lambda_1, \dots , \lambda_s  \in k \setminus \{ 0 \}$ such that   
\[
        f =   \sum_{i=1}^s \lambda_i g_i \quad \quad \text { and } \quad \quad
  g_1 > g_2  > g_3 > \dots > g_s.
\]
By definition,  $ \interm_{>} (f) = g_1$.  

\begin {remark}\label{rem!aboutinittermnpo1}
By the definition of a monomial ordering, we have 
\[
     \interm_{>} (f_1f_2) = (\interm_{>} (f_1))   (\interm_{>} (f_2))
\]
for all $f_1,f_2 \in  R \setminus \{ 0 \} $. 
\end {remark}

Moreover, by the definition of a monomial ordering   
we have the following proposition.

\begin{proposition} \label{prop!proposit1aboutinittermofsums}
Assume $f_1, f_2, \dots ,f_t \in   R \setminus \{ 0 \} $.
Assume  there exists $a$ with $1 \leq a \leq t$   such that 
\[
       \interm_{>} (f_a)  \;   >  \;  \interm_{>} (f_b)  
\]
for all  $b$ with $1 \leq b \leq t$ and $b \not= a$.  
Then   $\;   \sum_{i=1}^{t }f_i \not= 0$ and
$ \;   \interm_{>} (\sum_{i=1}^{t }f_i)   =  \interm_{>} (f_a)$.   \\
\end{proposition}

\begin{corollary} \label{cor!aboutinittermofsums}
Assume $f_1, f_2, \dots ,f_t \in   R \setminus \{ 0 \} $ satisfy
\[
       \interm_{>} (f_i)   \; \not=   \; \interm_{>} (f_j)  
\]
for all $1 \leq i,j \leq t$ with $i \not= j$. 
Then   $ \;  \sum_{i=1}^{t }f_i \not= 0$.
\end{corollary}

\begin {proof}
Since $ \interm_{>} (f_i)   \; \not=   \; \interm_{>} (f_j)$
for all $1 \leq i,j \leq t$ with $i \not= j$,
there exists a  unique integer $a$ such that
$1 \leq a \leq t$ and  $\interm_{>} (f_a)   >  \interm_{>} (f_b) $
for all $b$ with  $1 \leq b \leq t$ and
$b \not= a$. The result follows by
Proposition~\ref{prop!proposit1aboutinittermofsums}.
\end {proof}

\section{Lefschetz Properties and base change} 

We believe that the statements in the present section,
with the likely exception of Proposition~\ref{prop!inevensocledegreeslpandanisotimplyslp},
 are well-known. We include them for completeness.

\begin {proposition}  \label{prop!notvanishingoninfinitecylinders}
Assume $E$ is an infinite field, $f \in E[x_1, \dots ,x_m]$ is a nonzero polynomial
and, for  $1 \leq i \leq m$, $Z_i$ is an infinite subset of $E$. 
Then, there exists a point $p$  in the set  $Z_1 \times Z_2 \times \dots \times Z_m$
such that $f(p) \not= 0$.
\end {proposition}

\begin{proof}  We use induction on $m$. If $m=1$, it is well-known
that the polynomial $f$ has a finite number of roots in the field $E$,
and the result follows.

   Assume $m \geq 2$ and that the result is true for $m-1$. There exist
$s > 0$ and, for $0 \leq i \leq s$, a polynomial $g_i \in  E[x_1, \dots ,x_{m-1}]$,
such that 
\[
        f = \sum_{i=0}^s g_i x_m^i.
\]
Since $f$ is nonzero, there exists $c$, with $0 \leq  c \leq s$,
such that  $g_c$ is nonzero. Hence,
by the inductive hypothesis, there exists an element 
$\; (a_1, \dots , a_{m-1}) \in   Z_1 \times Z_2 \times \dots \times Z_{m-1}$  
such that  $g_c(a_1, \dots , a_{m-1}) \not= 0$. Consequently, the polynomial
$\;h \in E[x_m]$, with
\[
        h = \sum_{i=0}^s g_i(a_1, \dots , a_{m-1}) x_m^i,
\]
is nonzero. By the case $m=1$,  there exists $a_m \in Z_m$
such that $h(a_m) \not=0$. This implies that $f(a_1, \dots , a_m) \not= 0$.
\end{proof}

\begin {corollary}  \label{cor!notvanishingoninfinitesubfield}
Assume $E$ is an infinite field,  $m \geq 1$  is an integer
and  $f \in E[x_1, \dots ,x_m]$  is a nonzero polynomial.
Assume $k_1$ is an infinite subfield of $E$.  Then

i) There exists a point $p \in k_1^m$  such that $f(p) \not= 0$. 

ii)  Endow the set $E^m$ with the Zariski topology.
      Then the subset  $k_1^m$ of   $E^m$ is Zariski dense.
\end {corollary}

\begin{proof}   Part i) follows from 
Proposition~\ref{prop!notvanishingoninfinitecylinders}, by 
setting $Z_i = k_1$ for all $1 \leq i \leq m$.

Part ii) follows immediately from Part i).  
\end{proof}

Assume  that $k_1 \subset E$ is a field extension. We consider the polynomial 
ring  $k_1[x_1, \dots , x_m]$, where the degree of the variable $x_i$ is equal to $1$,
for all $1 \leq i \leq m$.
Assume   $I \subset k_1[x_1, \dots , x_m]$ 
is a homogeneous ideal  such that
the quotient  $G = k_1[x_1, \dots , x_m]/I$ is Cohen-Macaulay.   
We denote by $d$ the Krull dimension of $G$.

We set
$G_E = E[x_1, \dots , x_m]/(I)$, where $(I)$
is  the ideal of $E[x_1, \dots , x_m]$ generated by $I$.
By \cite [Theorem~2.1.10]{BH},  $G_E$ is also Cohen-Macaulay. Since, for all $i \geq 0$,
$(G_E)_i = G_i \otimes_{k_1} E$,  the Hilbert function of $G$ as a graded $k_1$-algebra is 
equal to the Hilbert function of $G_{E}$ as a graded $E$-algebra.
Consequently,  the  Krull dimension of $G_{E}$ is $d$.

\begin {proposition}  \label{prop!aboutAvsAbig}
  Assume that the field $k_1$ is infinite. Then the following are equivalent:

i) The graded $k_1$-algebra $G$ has the Weak Lefschetz Property.

ii) The graded  $E$-algebra $G_{E}$ has the Weak Lefschetz Property.
\end {proposition}

\begin{proof} 
 We first assume that $G$ has the Weak Lefschetz Property.  Then,
there exist elements
$g_1, \dots , g_d, \omega \in G_1$ such that 
$g_1, \dots , g_d$ is a regular sequence for $G$ and
$\omega$ is a Weak Lefschetz element for 
$G/(g_1, \dots, g_d)$.  Clearly, 
 $g_1, \dots , g_d$ is a regular sequence also for $G_{E}$ and
$\omega$ is a Weak Lefschetz element also for 
$G_{E}/(g_1, \dots, g_d)$.
Hence, the  $k$-algebra $G_{E}$ has the Weak Lefschetz Property.

For the opposite direction, we  assume that $G_{E}$ has the Weak Lefschetz Property. 
By taking the coefficients of $f_i$ and $\omega$, we can identify
the set   
\[
       \mathcal {S} =  \{  (g_1,  \dots , g_d, \omega) :  
                   g_i \in (G_{E})_1, \omega \in (G_{E})_1 \}
\]
with the affine space $(G_{E})_1^{d+1}$.  
We denote by $U$ the subset of $\mathcal {S}$ consisting of the element $ (g_1,  \dots , g_d, \omega)$
such that  $g_1, \dots , g_d$ is a regular sequence for $G_{E}$ and
$\omega$ is a Weak Lefschetz element for 
$G_{E}/(g_1, \dots, g_d)$.

By the assumption that  $G_{E}$ has the Weak Lefschetz Property, 
the set $U$ is nonempty.  
Hence, by \cite[Lemma~4.1]{BN},  $U$ is a nonempty Zariski 
open subset  of  $\mathcal {S}$.  
Using that the field $k_1$ is infinite, 
Corollary~\ref{cor!notvanishingoninfinitesubfield} implies that  $G_1^{d+1}$ is
Zariski dense in $(G_{E})_1^{d+1}$, hence $G_1^{d+1} \cap U  \not=  \newemptyset$.
Let $ (g_1,  \dots , g_d, \omega) \in G_1^{d+1} \cap U $.  Then 
$g_1, \dots , g_d$ is a regular sequence for $G$ and
$\omega$ is a Weak Lefschetz element for 
$G/(g_1, \dots, g_d)$.
 Hence, the  $k_1$-algebra $G$ has the Weak Lefschetz Property.
\end{proof}

We denote by $k$ the field of fractions  of the polynomial
ring
\[
             k_1 [ a_{i,j} :  1 \leq i \leq d,  \;  1 \leq j \leq m ].
\]
We set
$G_k = k[x_1, \dots , x_m]/(I)$, where $(I)$
is  the ideal of $k[x_1, \dots , x_m]$ generated by $I$.
For $1 \leq i \leq d$, we set $\; f_i = \sum_{j=1}^{m} a_{i,j} x_j$.
Hence,  the Artinian   $k$-algebra $G_{k}/(f_1, \dots , f_d)$ 
is    the generic Artinian reduction of the $k_1$-algebra $G$ in the sense of 
Definition~\ref{dfn!genericartinianreduction}.

\begin {proposition}  \label{prop!threeequivalencesforWLP}
 Assume $d \geq 1$.  Then the following are equivalent:  

i) The Artinian   $k$-algebra $G_{k}/(f_1, \dots , f_d)$ has the Weak Lefschetz Property.

ii)   If $E$ is an infinite field containing $k_1$ as a subfield, then the
      $E$-algebra $G_{E}$ has the Weak Lefschetz Property.

iii)   There exists an infinite field  $F$  containing $k_1$ as a subfield such that the
      $F$-algebra $G_{F}$ has the Weak Lefschetz Property.

\end {proposition}

\begin{proof} 
We first prove that i) implies ii).
Since   the  $k$-algebra $G_{k}/(f_1, \dots , f_d)$ has the Weak Lefschetz Property, it follows
that  the  $k$-algebra $G_{k}$ has the Weak Lefschetz Property.
Assume $E$  is an infinite field containing $k_1$ as a subfield.  We denote by  $E_1$ 
the field of fractions  of the polynomial ring
\[
             E [ a_{i,j} :  1 \leq i \leq d,  \;  1 \leq j \leq m ].
\]
Since $k$ is a subfield of $E_1$,      Proposition~\ref{prop!aboutAvsAbig}
implies that the  $E_1$-algebra $G_{E_1}$ has the  Weak Lefschetz Property.  Since
$E$ is an infinite subfield of $E_1$, the same proposition gives that the
$E$-algebra $G_{E}$ has the  Weak Lefschetz Property.

We now prove that ii) implies iii).  It is clear.

We now prove that iii) implies i).  
We denote by  $E$ the 
field of fractions  of the polynomial ring in one variable  $k[T]$ over $k$.
We denote by $F_1$  
the field of fractions  of the polynomial ring
\[
             F [ T,  a_{i,j} :  1 \leq i \leq d,  \;  1 \leq j \leq m ].
\]
Since $F$ is a subfield of $F_1$, both fields are infinite,  and,  by the 
assumption,  the $F$-algebra $G_{F}$ has the Weak Lefschetz Property,
it follows, by Proposition~\ref{prop!aboutAvsAbig},  
that the  $F_1$-algebra $G_{F_1}$ has the Weak Lefschetz Property. 
Since $E$ is  an infinite subfield of $F_1$, the same proposition 
implies that the  $E$-algebra $G_{E}$ has the  Weak Lefschetz Property.

We denote by $I^e$ the ideal of $E[x_1, \dots , x_m]$  generated by $I$,
and by $V$
 the $m$-dimensional $E$-vector subspace of  $E[x_1, \dots , x_m]$ 
consisting of homogeneous degree $1$ polynomials.
For   $1 \leq i \leq d,  \;  1 \leq j \leq m$, we define the infinite subset 
\[
     Z_{i,j} =  \{  a_{i,j} +  T^r \;  :  \; r \geq 1  \}
\]
of $E$.    We denote by $Z$ the Cartesian product, 
for  $1 \leq i \leq d,  \;  1 \leq j \leq m$, of
the sets $Z_{i,j}$.
By Corollary~\ref{cor!notvanishingoninfinitesubfield},
$Z$  is Zariski dense in the affine space $E^{dm}$.

Since $G_{E}$ has the   Weak Lefschetz Property, the set $U$ 
consisting  of  all $ (g_1,  \dots , g_d) \in  V^{d}$  such that  
$g_1,  \dots ,  g_d$ is a regular sequence for $G_{E}$ 
and $G_{E}/ (g_1,  \dots , g_d)$ has 
the   Weak Lefschetz Property,
is a nonempty
Zariski open subset of the affine space    $V^{d}$.

We identify $V^{d}$  with $E^{dm}$,  by  considering the coefficients
of the  homogeneous degree $1$ polynomials.
Since $Z$  is Zariski dense in $E^{dm}$, 
the intersection of $Z$ with
$U$ is nonempty.  Hence, for  $1 \leq i \leq d,  \;  1 \leq j \leq m$, 
there exists a positive integer $r_{i,j}$ such that, if we set 
\[
               g_i  =   \sum_{j=1}^{m} ( a_{i,j} +  T^{r_{i,j}}  ) x_j,
\] 
we have that  $g_1,  \dots , g_d$ is a regular sequence for $G_{E}$ 
and $G_{E}/ (g_1,  \dots , g_d)$ has the   Weak Lefschetz Property.

There exists a  unique $k_1$-linear  automorphism  of the polynomial ring
$k_1[a_{i,j}, T]$  that sends $T$ to $T$ and  $a_{i,j}$  to  $a_{i,j} +  T^{r_{i,j}}$,
for all $i,j$.  The automorphism extends first to a field automorphism of $E$
and then to a degree preserving  automorphism $\phi$ of the polynomial ring $E[x_1, \dots , x_m]$
that sends $x_i$ to $x_i$, for all $1 \leq i \leq m$, and is the identity when
restricted to $k_1$.  
Hence,    $\phi(I^e) = I^e$ and $\phi(f_i) = g_i$, for all
$1 \leq i \leq d$, which imply that
\[
            \phi ( I^e + (f_1, \dots , f_d)) =  I^e + (g_1, \dots , g_d).
\]
Consequently,    $f_1,  \dots , f_d$ is a regular sequence for $G_{E}$ 
and $G_{E}/ (f_1,  \dots , f_d)$ has the   Weak Lefschetz Property,
since the same properties hold for  $g_1,  \dots , g_d$ 
and $G_{E}/ (g_1,  \dots , g_d)$.  

Finally, since $k$ is an infinite subfield of $E$, 
Proposition~\ref{prop!aboutAvsAbig}
implies that the  $k$-algebra $G_{k}/ (f_1,  \dots , f_d)$ has the  Weak Lefschetz Property.
\end{proof}

We now discuss  the corresponding  statements of the last two propositions for the
Strong Lefschetz Property.

\begin {proposition}  \label{prop!aboutAvsAbigforSLP}
  Assume that the field $k_1$ is infinite and $G$ is
    Gorenstein.  Then the following are equivalent:

i) The graded $k_1$-algebra $G$ has the Strong Lefschetz Property.

ii) The graded  $k$-algebra $G_{E}$ has the Strong Lefschetz Property.
\end {proposition}

\begin{proof}  With the obvious modifications, the arguments in the proof of
Proposition~\ref{prop!aboutAvsAbig}  also work  here.
\end{proof}

\begin {proposition}  \label{prop!threeequivalencesforSLP}
  Assume that $G$ is Gorenstein and $d \geq 1$.
    Then the following are equivalent:  

i) The Artinian   $k$-algebra $G_{k}/(f_1, \dots , f_d)$ has the Strong Lefschetz Property.

ii)   If $E$ is an infinite field containing $k_1$ as a subfield, then the
      $E$-algebra $G_{E}$ has the Strong Lefschetz Property.

iii)   There exists an infinite field  $F$  containing $k_1$ as a subfield such that the
      $F$-algebra $G_{F}$ has the Strong Lefschetz Property.

\end {proposition}

\begin{proof}  With the obvious modifications, the arguments in the proof of
Proposition~\ref{prop!threeequivalencesforWLP}  also work  here.
\end{proof}

We also need the following two propositions.
The first  is a special case of Part (b) of
\cite[Proposition~2.1]{MiMRN}. 

\begin{proposition} \label{prop!injectivityforsmallerdegerees}
Assume $k_1$ is a field and  $A$ is a standard graded Artinian Gorenstein $k_1$-algebra
of socle degree $d$. Assume $s$ in an integer with $1 \leq s < d$. Assume
$\omega \in A_1$ has the property that the multiplication by $\omega$
map  $A_{s} \to A_{s+1}$  is injective. Then, for all
$t$ with $0 \leq t \leq s$, we have  that the multiplication by $\omega$
map  $A_{t} \to A_{t+1}$  is injective.
\end{proposition}

\begin{proof}
   Assume $0 \leq t \leq s$ and
$0 \not= u \in A_t$. By Remark~\ref{rem!poincare_duality_for_gor},
there exists $z \in A_{s-t}$ such that $uz \not= 0$. Hence
$\omega (u z) \not= 0$, which implies that  $ \omega u \not= 0$.
\end{proof}

\begin{proposition} \label{prop!inevensocledegreeslpandanisotimplyslp}
Assume $k_1$ is a field and  $A$ is a standard graded Artinian Gorenstein $k_1$-algebra
of even socle degree $d$.  We assume that $A$ has the Weak Lefschetz 
Property and that, for all $i$ with  $0 \leq i \leq d/2$ and  all
$0 \not= u \in A_i$,  
we  have $u^2 \not=0 $.  Then $A$ has the Strong Lefschetz  Property.
\end{proposition}

\begin{proof}
We fix $\omega \in A_1$  such that, for all $t \geq 0$,
the multiplication by $\omega$ map
$A_t \to A_{t+1}$ has maximal rank.
Since $A$ is Gorenstein of even socle degree $d$, it follows that
the multiplication by $\omega$ map form $A_{d/2 -1} \to A_{d/2}$ 
is injective.  By the definition of the  Strong Lefschetz  Property, 
and using that $A$ is Gorenstein,
to prove the proposition it is enough to show that
for all $i$,  with $0 \leq i < d/2$,
the multiplication by $\omega^{d-2i}$ map 
from $A_{i}$ to $A_{d-i}$ is injective.

Assume  $0 \leq i < d/2$  and  that  $z \in A_i$ has the property
\[
                     \omega^{d-2i} z = 0.
\]
As a consequence,  $\omega^{d-2i} z^2 = 0$. Using the assumption,
it follows that  $\omega^{d/2-i} z = 0$.  
Proposition~\ref{prop!injectivityforsmallerdegerees}
implies that $z = 0$.
\end{proof}

\section{A conjecture about differentiation} 

Assume  $k_1$ is a field of characteristic $2$.
Assume $n \geq 1$ is an integer and $D$ is a simplicial sphere of dimension
 $n$ with vertex set  $ \{1, 2,  \dots, m \}$.
We denote by $k$ the field of fractions of the polynomial
ring
\[
             k_1 [ a_{i,j} :  1 \leq i \leq n+1,  \;  1 \leq j \leq m ].
\]
We define the polynomial ring  $R = k[x_1, \dots , x_m]$,
where we put degree $1$ for all variables $x_i$. 
We denote by $I_D \subset R$ the
Stanley-Reisner ideal of $D$ and we set $k[D]=R/I_D$.  
For $i=1, \dots , n+1$, we set 
\[
         f_i = \sum_{j=1}^{m} a_{i,j} x_j,
\]
and we define $A = k[D]/(f_1, \dots , f_{n+1})$.
Hence, $A$ is the generic Artinian reduction of $k_1[D]$ in the sense of 
Definition~\ref{dfn!genericartinianreduction}.
We denote by $ \pi :  R \to A$ the natural projection $k$-algebra
homomorphism, and by $\Psi : A_{n+1} \to k$ the vector space 
isomorphism defined in Remark~\ref{rem!signofPsi}.

For a finite sequence $\; \delta= (\delta_1, \dots , \delta_{n+1}) \;$ 
such that $1 \leq \delta_i \leq m$ for all $1 \leq i \leq n+1$
we set 
\[
            x_{\delta} = \prod_{i=1}^{n+1} x_{\delta_i} \in R.
\]

Assume $\; \sigma= (\sigma_1, \dots , \sigma_{n+1}) \;$ and
$\; \tau = (\tau_1, \dots , \tau_{n+1})\;$  are two finite sequences 
such that $1 \leq \sigma_i, \tau_i \leq m$, for all $1 \leq i \leq n+1$.
We denote by $\; \partial_{\sigma}^{mod} :  k \to k\;$ the 
$(n+1)$-th order differential operator
which is differentiation with respect to the variables in the set
\[
                \{ a_{i, \sigma_i}   \;:    \; 1 \leq i \leq n+1  \}.
\]
The following conjecture, if true, will generalise	 Theorem~\ref{theorem!computingdifferoperators}
and Propositions~\ref{prop!differoddnpfofmainthm}  and~\ref{prop!differforevennpfofmainthm}.

\begin {conjecture} \label{conj!generalisingthmcomputingdifferoperators}
  1)  Assume the monomial $\; x_{\sigma} x_{\tau} \;$
is not the square of a monomial in $R$.  We then have
\[
        (\partial_{\sigma}^{mod} \circ \Psi  \circ   \pi)  (x_\tau)   =  0.
\]

2)   Assume $\; x_{\sigma} x_{\tau} \;$
is the square of a monomial in $R$. Assume 
 $\; \delta = (\delta_1, \dots , \delta_{n+1} )$ is a finite sequence
such that $\;1 \leq \delta_i \leq m$, for all $1 \leq i \leq n+1$, and
 $\;   x_{\sigma} x_{\tau} = (x_{\delta})^2$.
We then have
\[
     (\partial_{\sigma}^{mod} \circ \Psi  \circ   \pi)  (x_\tau) 
          = \big((\Psi  \circ   \pi)  (x_\delta) \big)^2.
\]
\end {conjecture}

\begin{remark} We note that  Conjecture~\ref{conj!generalisingthmcomputingdifferoperators}   
implies the following interesting equality
\[
     (\partial_{\sigma}^{mod}  \circ \Psi  \circ   \pi)  (x_\tau)   =
                 (\partial_{\tau}^{mod}  \circ \Psi  \circ   \pi)  (x_\sigma). 
\]
\end{remark}

\begin{remark} Assume $i \geq 1$ and   $t \geq 2$  are two integers such that
$\; ti \leq n+1$. Assume $0 \not= u \in A_i$.
Theorem~\ref{thm!anisotropicityinchar2} implies that if $t$ is a
power of $2$ then  $u^t \not= 0$.  Is this also true  for
all values of $t$?
\end{remark}

\section* {Acknowledgements} 

\label{sec!acknowledgements}
S.~A.~P.  thanks Christos  Athanasiadis for suggesting the problem,
and David Eisenbud   for useful conversations.
We benefited from
experiments with the computer algebra program Macaulay2~\cite{GS}. 
This work is part of the Univ. of Ioannina Ph.D. thesis of V. P., financially 
supported by the Special Account for Research
Funding (E.L.K.E.) of  the University of Ioannina under the program with code number
82561 and title \q{Program of financial support for Ph.D. students and postdoctoral researchers}.


\begin{thebibliography}{99}

\bibitem{A1}
\textsc{K. Adiprasito},
Combinatorial Lefschetz theorems beyond positivity,
arXiv preprint,  2018, 76 pp., available at  \href{https://arxiv.org/abs/1812.10454}{https://arxiv.org/abs/1812.10454}.


\bibitem{A2}
\textsc{K. Adiprasito},
FAQ on the g-theorem and the hard Lefschetz theorem for face rings. 
\emph{Rend. Mat. Appl.}  (7) 40 (2019),  97--111.



\bibitem{BN}
\textsc{E. Babson and E. Nevo},
Lefschetz properties and basic constructions on simplicial spheres. 
\emph{J. Algebraic Combin.}
J. Algebraic Combin. 31 (2010), 111--129.


\bibitem{BL}
\textsc{L. J. Billera and C. W. Lee},
A proof of the sufficiency of McMullen's conditions for f-vectors of simplicial convex polytopes. 
\emph{J. Combin. Theory Ser. A}  31 (1981),  237--255. 


\bibitem{BH}
\textsc{W. Bruns and J. Herzog},
Cohen-Macaulay rings.
Cambridge Studies in Advanced Mathematics, 39. Cambridge University Press, Cambridge, 1993.


\bibitem{BP}
\textsc{V. M.  Buchstaber and T. E. Panov},
Torus actions and their applications in
topology and combinatorics, 
University Lecture Series, 24.
American Mathematical Society,  2002, viii+144 pp.



\bibitem{Ei}
\textsc{D. Eisenbud},
 Commutative algebra. With a view toward
algebraic geometry. Graduate Texts in Mathematics, 150. Springer-Verlag, New
York, 1995.

\bibitem{GS} 
\textsc{D. Grayson and M. Stillman},
 Macaulay2, a software system
for research in algebraic geometry, available at
\href{http://www.math.uiuc.edu/Macaulay2/}{http://www.math.uiuc.edu/Macaulay2/}


\bibitem{HMetal}
\textsc{T. Harima,  T. Maeno, H. Morita, Y. Numata,
A. Wachi and J. Watanabe},
The Lefschetz properties. Lecture
Notes in Mathematics, 2080. Springer, Heidelberg, 2013. xx+250 pp.



\bibitem{LB} 
\textsc{V. Lakshmibai and J. Brown},
The Grassmannian variety.
Geometric and representation-theoretic aspects.
Developments in Mathematics, 42.
Springer, New York, 2015. x+172 pp.


\bibitem{Lee} 
\textsc{C. W. Lee}, 
Generalized stress and motions. 
Polytopes: abstract, convex and computational (Scarborough, ON, 1993), 249--271,
NATO Adv. Sci. Inst. Ser. C Math. Phys. Sci., 440, Kluwer Acad. Publ., Dordrecht, 1994.


\bibitem{MiMRN}
\textsc{J. Migliore, R. Mir\'o-Roig  and U. Nagel},
Monomial ideals, almost complete intersections and the Weak Lefschetz Property,
\emph{Trans. Amer. Math. Soc.} 363 (2011), 229--257.


\bibitem{MiNa}
\textsc{J. Migliore and U. Nagel},
 Survey article: a tour of the weak
and strong Lefschetz properties, \emph{J. Commut. Algebra} 5 (2013), 329--358.





\bibitem{MZ}
\textsc{J. Migliore and F. Zanello},
The strength of the weak
Lefschetz property, \emph{Illinois J. Math.} 52 (2008), 1417--1433.


\bibitem{St3}
\textsc{R. Stanley},
The number of faces of a simplicial convex
polytope,  \emph{Adv. in Math.} 35 (1980), 236--238.


\bibitem{St1}
\textsc{R. Stanley},
Combinatorics and commutative algebra, Second
edition. Progress in Mathematics, 41. Birkh\"{a}user, 1996.
 

\bibitem{S1} 
\textsc{E. Swartz},
Face enumeration - from spheres to manifolds. 
\emph{J. Eur. Math. Soc.}  11 (2009),  449--485.


\bibitem{Swa}
\textsc{E. Swartz},
Thirty-five years and counting, arXiv preprint, 2014, 29 pp., available at \href{https://arxiv.org/abs/1411.0987v1}{arXiv:1411.0987v1}



\end{thebibliography}
\end{document}